\documentclass[12pt]{amsart}
\usepackage[hyphens]{url} 
\usepackage[colorlinks,allcolors=blue]{hyperref}
\usepackage{amssymb}
\usepackage{amsthm}
\usepackage{amsfonts}
\usepackage{mathtools}
\usepackage{amsmath}
\usepackage[all,cmtip]{xy}
\usepackage{fullpage}
\usepackage{mathrsfs}
\usepackage{graphicx}
\usepackage{tensor}
\usepackage{tikz}
\usepackage{tikz-cd}
\usepackage{quiver}
\usetikzlibrary{matrix}
\usetikzlibrary{arrows}
\usepackage{float}

\usepackage[style=alphabetic, doi=false, url=false]{biblatex}
\usepackage{kimballstyle}

\newcommand{\catsus}{{\Sigma}}

\newcommand{\cof}{\mathrm{cof}}

\renewcommand{\tilde}{\widetilde}
\tikzset{
    labl/.style={anchor=south, rotate=90, inner sep=.5mm}
}

\usepackage[textsize = tiny, disable]{todonotes}

\begin{document}

\title{Change of Enrichment for Monoidal Model Categories}
\author{Kimball Strong}
\address{Department of Mathematics, Cornell University, Ithaca, NY 14853}
\email{ks2424@cornell.edu}
\date{}
\begin{abstract}
	A classical theorem of category theory says that given an adjunction $L: \V \rightleftarrows \W: R$ between monoidal categories with lax monoidal right adjoint, there is an induced adjunction $\lcat: \vcat \rightleftarrows \wcat : \rcat$ between categories of enriched categories.
	We extend this result to monoidal model categories (with some model categorical assumptions), improving on previous theorems which required the derived functor of $L$ to be strong monoidal.
	As an application, we consider the category of chain complexes over a varying groupoid of operators, equipped with two different monoidal structures: the standard tensor product, and the cartesian product. 
	We show that both of these are monoidal model categories, and use our theorems to reproduce the square-zero extensions adjunction for augmented $\dg$-categories.
\end{abstract}
\maketitle
\tableofcontents

\section*{Introduction}
Suppose $\V$ and $\W$ are monoidal categories, and that 
$$L: \V \rightleftarrows \W: R $$
is an adjunction of underlying categories.
If $L$ is ``strong monoidal,'' that is, if there is a natural isomorphism $L(X \otimes Y) \cong L(X) \otimes L(Y)$, then this adjunction defines an adjunction in a $2$-category of monoidal categories, and it then follows that it induces a change-of-enrichment adjunction
$$\lcat: \V\cat \rightleftarrows \W\cat: \rcat$$
which is given by applying $L$ and $R$ locally (i.e. by applying $L$ and $R$ to the hom-objects).
This was generalized to homotopical settings by Muro in \cite{Muro2012} in which the corresponding statement for Quillen adjunctions is proved: that is, if $\V$ and $\W$ are monoidal model categories, and
$$L: \V \rightleftarrows \W: R $$
an adjunction such that $L$ is strong monoidal ``up-to-weak-equivalence'' (i.e. there is a natural weak equivalence $L(X \otimes Y) \simeq L(X) \otimes L(Y)$), then (under some mild model categorical assumptions) Muro shows there is a change-of-enrichment Quillen adjunction
$$\lcat: \V\cat \rightleftarrows \W\cat: \rcat$$
which (up to weak equivalence) is given by applying $L$ and $R$ locally.
\par
The setting of a strong monoidal left adjoint is the most common setting in which to consider a change-of-enrichment adjunction, but it is a much stronger hypotheses than is actually needed. 
More generally, if $L: \V \rightleftarrows \W: R $ is an adjunction between monoidal categories and the right adjoint $R$ is lax monoidal, meaning that there is a natural transformation (not necessarily, and indeed not generally an isomorphism) 
$$R(X) \otimes R(Y) \to R(X \otimes Y)$$
then there exists (under mild categorical assumptions) an induced change-of-enrichment adjunction
$$\lcat: \vcat \rightleftarrows \wcat: \rcat$$
Where the right adjoint $\rcat$ is given by applying $R$ locally (i.e. to each hom-object), but the left adjoint $L^\mathrm{cat}$ is not.
\par 
The main purpose of the present paper is to transfer this result into the homotopical setting. Before stating our main theorem, we need some amount of terminology: recall that for $\V$ a monoidal category, the category $\vcat$ is equipped with a canonical functor $\pi_0 : \vcat \to \cat$, which sends $\C \in \vcat$ to the category $\pi_0\C$ which has the same objects as $\C$, and for which
$$\pi_0\C[x,y] = \Hom_\V(\iv, \C[x,y])$$
where $\iv$ is the unit of $\V$. 
We can extend this to the model categorical setting; if $\V$ is a model category, then we can define $\pi_0: \vcat \to \cat$ by setting 
$$\pi_0\C[x,y] = \Hom_{\Ho(\V)}(\iv, \C[x,y])$$
\begin{definition}
	Let $\V$ be a model category equipped with a symmetric monoidal structure. We call a functor $f: \C \to \D$ of $\V$-categories a \emph{Dwyer-Kan equivalence} if it satisfies:
	\begin{itemize}
		\item \textbf{(homotopical full faithfulness):} for any $x,y \in \C$, the map $\C[x,y] \to \D[fx,fy]$ is a weak equivalence in $\V$.
		\item \textbf{(homotopical essential surjectivity):} the induced functor $\pi_0f:\pi_0\C \to \pi_0\D$ is essentially surjective. 
	\end{itemize} 
	Note that these together imply that $\pi_0f$ is in fact an equivalence of categories, as homotopical full faithfulness implies that the map 
	$$\Hom_{\Ho(\V)}(\iv, \C[x,y]) \to \Hom_{\Ho(\V)}(\iv, \D[fx,fy])$$ 
	is an isomorphism. 
\end{definition}
\begin{definition}
	Let $\V$ be a model category equipped with a closed symmetric monoidal product. The \emph{Dwyer-Kan model structure} on $\vcat$ is (if it exists) the model structure in which:
	\begin{itemize}
		\item The weak equivalences are the Dwyer-Kan equivalences, and
		\item The trivial fibrations are the local trivial fibrations which are surjective on objects.
	\end{itemize}
\end{definition}
Note that this model structure need not always exist, although both \cite{Muro2012} and \cite{BergerMoerdijk2012} give fairly general conditions on $\V$ under which it does.
Our first main theorem is:
\begin{alphatheorem}[Theorem \ref{thm:main-theorem-chapter-two}]\label{alphathm:complicated}
	Let $\V$ and $\W$ be symmetric monoidal model categories satisfying the model-categorical assumptions (\cite{Muro2012}) to construct the Dwyer-Kan model structures on $\vcat$ and $\wcat$.
	Let $L: \V \rightleftarrows \W: R$ be a Quillen adjunction with $R$ lax monoidal.
	Furthermore suppose that the colax structure map $L(\iv) \to \iw$ is a weak equivalence.
	Under some model categorical assumptions on $\V$, the induced adjunction
	\adjunctiondiagram{\catfunctor{L}}{\vcat}{\wcat}{\catfunctor{R}}
	is a Quillen adjunction.
\end{alphatheorem}
The condition that $L$ weakly preserves the unit is more restrictive than the classical (non model-categorical) theorem, but as we note in Section \ref{section:necessity-lax-unitality}, the classical theorem is somewhat misleading: even when $\V$ and $\W$ are ordinary categories, $\vcat$ and $\wcat$ are naturally relative categories (with the Dwyer-Kan equivalences). 
We give an example to show that without imposing some condition on the the colax structure map $L(\iv) \to \iw$, the induced adjunction $\lcat: \vcat \rightleftarrows \wcat : \rcat$ may be badly behaved, e.g. may not preserve Dwyer-Kan equivalences. 
\par 
Our second main theorem is both a separate set of hypotheses with which to prove that $\lcat : \vcat \rightleftarrows \wcat : \rcat$ is a Quillen adjunction, as well as a tool for analyzing the Dwyer-Kan model structure.
One difficulty analyzing this model structure is that the fibrations can \textit{a priori} be quite complex. 
\begin{definition}\label{dfn:naive-fibration}
	Let $\V$ be a model category equipped with a monoidal structure. 
	We shall refer to the maps in $\vcat$ which are local fibrations and induce isofibrations on homotopy categories as the \textit{naive fibrations.}
\end{definition}
Under the hypotheses of \cite{BergerMoerdijk2012}, fibrations in $\vcat$ must be in particular naive fibrations, but there may be naive fibrations which are not fibrations. 
Fortunately, in practice these two classes often coincide.
\begin{definition}\label{dfn:naive-fibration-property}
	We shall say that a model structure on $\vcat$ has the \emph{naive fibration property} if the fibrations between fibrant objects are exactly the local fibrations which induce isofibrations of homotopy categories.
\end{definition}
In practice, many examples of interest have the naive fibration property. 
For instance, simplicial categories \cite{BergnerModelStructure} and $\dg$-categories \cite{Tabuada2004} both have the naive fibration property. 
Lurie proves that for a general class of $\V$ (satisfying hypotheses somewhat dual to those of Berger and Moerdijk) $\vcat$ has the naive fibration property \cite{LurieHTT}.
Our second main theorem can be roughly stated as follows:
\begin{alphatheorem}[Theorem \ref{thm:naive-fibration-transfer}]\label{alphathm:simple}
	Let $\V$ and $\W$ be symmetric monoidal model categories satisfying the model-categorical assumptions (\cite{BergerMoerdijk2012}) to construct the Dwyer-Kan model structures on $\vcat$ and $\wcat$.
	 Let $L: \V \rightleftarrows \W: R$ be a Quillen adjunction with $R$ lax monoidal. Further suppose that $\vcat$ has the weak naive fibration property. 
	 Then if $R$ preserves naive fibrations, the adjunction
	 \adjunctiondiagram{\catfunctor{L}}{\vcat}{\wcat}{\catfunctor{R}}
	 is Quillen, and $\wcat$ has the weak naive fibration property. 
\end{alphatheorem}
This theorem is not particularly difficult to prove, as the assumption that $\vcat$ has the weak naive fibration property greatly simplifies checking that $\rcat$ preserves fibrations, which in geneneral is the difficult part of proving that the change of enrichment adjunction is Quillen. 
Fortunately, the weak naive fibration property is satisfied for many examples of interest, and the theorem gives a tool to ensure more examples satisfy this.
Indeed, we will show that all of the examples we consider here satisfy the weak naive fibration property.  
\par 
An example which we explore in detail here is when the underlying category of both $\V$ and $\W$ is the category of 
(potentially augmented) chain complexes with a groupoid of operators. 
These arise naturally in homotopy theory by considering the action of the fundamental groupoid of a space on the singular chain complex of its universal cover.
We show that the category $\chgpd$ of such objects (which is easily shown to be a model category) admits two monoidal products with respect to which it is a monoidal model category. 
The first is a tensor product that combines the product of groupoids with the tensor product of chain complexes; the fact that this is a monoidal model category is straightforward.
More interesting is the cartesian product: the category $\chbasic$ of ordinary chain complexes is not cartesian closed, and thus lacks a basic prerequisite to being called a ``monoidal model category.'' 
However, it follows from recent work of Nunes and V{\'a}k{\'a}r in \cite{Nunes2024} that $\chgpd$ is cartesian closed, and we furthermore show that it is a cartesian monoidal model category.
Passing to the category $\chaug$ of augmented objects makes the identity functor lax monoidal, giving an adjunction
\adjunctiondiagram{}{\trackdgcataug}{\trackchaincat}{}
between model categories of enriched categories. 
We use our second main theorem (Theorem \ref{thm:naive-fibration-transfer}) to prove that this is a Quillen adjunction, although notably our methods are still somewhat ad-hoc.
\par 
We refer to objects of $\trackdgcat$ as ``track-dg categories,'' as they combine track categories (categories enriched in groupoids) with dg-categories (categories enriched in $\chbasic$ with respect to the tensor product). 
One can think of them as being a $(2,1)$-category enhanced with the structure of a dg-category paramaterized by the $(2,1)$-categorical structure. 
Objects in $\trackchaincat$, which we refer to as \textit{track-chain categories} are less familiar, but much simpler than track-dg categories. 
They are best thought of as being a system of chain complexes paramaterized by the morphisms of a $(2,1)$-category. 
The most illuminating examples are single-object categories, i.e. monoids: considering only single-object categories with trivial groupoid part,
the left-hand side consists simply of augmented DGA's, whereas the right-hand side consists simply of chain complexes. The adjunction is the adjunction of indecomposables and square-zero extensions.
\par 
More generally, one may embed ordinary augmented $\dg$-categories into the left hand side of the adjunction by viewing a chain complex as having an action by the trivial groupoid. 
Then the output track chain-category will be equivalently a chain complex, and we prove that this construction recovers the non-commutative $1$-forms $\Omega^1_{nc}$ of \cite{Tabuada2009}, used to define the Andr{\'e}-Quillen homology of a $\dg$-category.
\par We would like to note that both theorems assume that the monoidal product is symmetric:
we believe that the assumption that $\V$ and $\W$ are symmetric monoidal is likely not necessary, but we have included it because the existing literature (in particular \cite{Muro2012} and \cite{BergerMoerdijk2012}) assumes this in constructing a model structure on $\vcat$. 
Historically, the monoidal model categories of interest in homotopy theory have been symmetric. 
However, recently there has been growing interest in the Gray tensor product, a nonsymmetric monoidal structure on $(\infty,n)$-categories.
Indeed, this paper is largely a companion paper to \cite{Strong2025}, in which the results are used to study the strictification adjunction between $(\infty,1)$-categories and $(\omega,1)$-categories, a more ``homological'' version of $(\infty,1)$-categories.
In that case, the gray tensor product (of $(\infty,0)$-categories, i.e. $\infty$-groupoids) is in fact symmetric, hence our results apply already.
In future work, we plan to extend the results here to the more general nonsymmetric case, but this will require first verifying that the Dwyer-Kan model structure can be constructed without the symmetry assumption on $\V$.

\subsection{Notation and Conventions}
	\noindent We assume the reader is familiar with the theory of model categories, including in particular basic homotopy colimits (we shall make use only of homotopy pushouts and directed homotopy colimits).
	\par 
	By ``monoidal model category'' $\V$ we will mean a monoidal category $(\V, \otimes_V, \iv)$ where the underlying category is equipped with a model structure, satisfying the following two compatibility axioms:
\begin{definition}
  A model category $\V$ equipped with monoidal product $\otimes$ satisfies the \emph{pushout-product axiom} if for any two cofibrations $i : A \to B$ and $j : C \to D$, the morphism
    \[
      i \boxtimes j: B \otimes C \amalg_{A \otimes C} A \otimes D \to
      B \otimes D,
    \]
    induced by the commutative square
    \[
      \xymatrix{ A \otimes C \ar[d]_{i \otimes C} \ar[r]^{A \otimes j}
        &
        A \otimes D \ar[d]^{i \otimes D} \\
        B \otimes C \ar[r]_{B \otimes j} & B \otimes D }
    \]
    is a cofibration. Furthermore, we require that if either $i$ or $j$ is a trivial
    cofibration, then so is $i \boxtimes j$.
\end{definition}
\begin{definition}
	A model category $\V$ equipped with monoidal product $\otimes$ and unit $\iv$ satisfies the \emph{unit axiom}
  	if, for every cofibrant replacement $p : Q\iv \to \iv$ of the tensor unit
    and every cofibrant object $A$, both $p \otimes A : Q\iv \otimes A \to \iv
    \otimes A$ and $A \otimes p : A \otimes Q\iv \to A \otimes \iv$ are weak
    equivalences.
\end{definition}
Furthermore, we are often interested in the \textit{monoid axiom}:
\begin{definition}[\cite{SchwedeShipley1998}]\label{dfn:monoid-axiom}
	Let $\V$ be a model category equipped with a monoidal structure $\otimes$. Let $R$ be the set of maps of the form $X \otimes f$ for $X$ any object of $\V$ and $f$ an acyclic cofibration.  
	We say that $\V$ satisfies the \emph{monoid axiom} if every morphism obtained as a transfinite composition of pushouts of maps in $R$ is a weak equivalence. 
\end{definition}
When there is no possible ambiguity, we will simply write $\otimes$ to be the monoidal product in $\V$. 
When an expression or statement involves multiple monoidal categories, we will use the notation $\ov$ and $\ow$ to distinguish the products.
\par
Similarly, we shall write $\vcat$ to denote the category of (small) categories enriched in $\V$; if we wish to make the monoidal structure more obvious we will write e.g. $\V^\otimes\text{-}\cat$.
One well-known example which we shall make use of is the following:
\begin{definition}
	A \emph{track-category} is a category enriched in the cartesian monoidal category $\gpd$ of groupoids. 
\end{definition}
We denote by $\T$ the free functor from $\V$ to $\mathrm{Mon}(\V)$ (the category of monoids in $\V$), left adjoint to the forgetful functor. This adjunction exists whenever $\V$ is closed monoidal.
\par 
We will not take care to distinguish between $\V$-categories with one object and monoids in $\V$; e.g. treating $\T$ as giving a functor $\V \to \vcat$. 
Similarly, given a set $S$, we will view it as being the discrete $\V$-category with object set $S$, in which
$$
	S[x,y] = 
\begin{cases}
	\emptyset & \text{ if } x \neq y \\
	\iv & \text{ if } x = y
\end{cases}
$$
where $\emptyset$ is the intial object of $\V$. 
We will use this mostly in the case where $S$ is a singleton set. 
\subsection{Organization of the Paper}
	We will begin in Section \ref{section:background} by reviewing some background on change of enrichment for ordinary (i.e., non-model) categories, as well as the model category case when the object set is fixed (which is significantly easier to analyze).
	In Section \ref{section:thmA} we prove the second of our change of enrichment theorems for monoidal model categories (Theorem \ref{alphathm:simple}). 
	In Sections \ref{section:chgpd} and \ref{section:chgpd-categories}, 
	we apply this theorem to the category of chain complexes with action by a (varying) groupoid,
	and in Section \ref{section:bar-construction} show that the resulting adjunction recovers
	the noncommutative $1$-forms of an augmented $\dg$-category. 
	In Section \ref{section:thmB}---by far the most technical section---we prove Theorem \ref{alphathm:complicated}.
	In Section \ref{section:necessity-lax-unitality} we give a ``counterexample'' to demonstrate the necessity of one of our seemingly extraneous hypotheses, showing that even when the model structures involved are trivial the induced adjunction can fail to be Quillen.

\subsection*{Acknowledgments}
	This paper is adapted from the second chapter of my PhD thesis, and so I would like to thank my advisor Inna Zakharevich for her guidance and for many discussions about the content.
	I would like to thank Varinderjit Mann for helping me understand some of the subtler bits of enriched category theory, and pointing me to many helpful references.
\section{Background: Change of Base for Monoidal Categories and Monoidal Model Categories with Fixed Object Set}
\label{section:background}
In this section we recall the classical (non-homotopical) facts about change of enrichment for monoidal categories, and define some notations we will need later. 
Most of this is borrowed from \cite[Sections 2,3]{Muro2012}, and the reader who wants slightly more detail should look there for it.
\par 
Let $\V$ and $\W$ be locally presentable biclosed monoidal categories, and let $L: \V \rightleftarrows \W : R$ be an adjunction of categories with $R$ lax monoidal. 
For $S$ a set, denote by $\vcat_S$ the category whose objects are $\V$-categories with fixed object set $S$, and whose morphisms are identity-on-objects $\V$-functors. Similarly for $\wcat_S$.
It follows via \cite[Theorem 4.5.6]{Borceux1994} that for any set $S$, there is an adjunction 
\adjunctiondiagram{\lcat_S}{\vcat_S}{\wcat_S}{\rcat_S}
In which $\rcat_S$ is given by applying $R$ to each hom object, using the lax monoidal structure of $R$ to produce a $\V$-category. 
Moreover, it follows by \cite[Corollary 1.11.3]{Giraud1971} that these adjunctions stitch together to give an adjunction
\adjunctiondiagram{\lcat}{\vcat}{\wcat}{\rcat}
While $\rcat$ is given by applying $R$ locally, $\lcat$ can be quite nontrivial to compute.
Essentially everything that we deduce about it will be by utilizing the fact that it is left adjoint to $\rcat$ and preserves the object set.
\par 
Bringing a model category structure into the picture is not too difficult provided one is willing to restrict to a fixed object set:
fixing a set of objects $S$, one can deduce that $\vcat_S$ is the category of monoids in $\vgraph_S$ for a monoidal product called the ``circle product.''
$\vgraph_S$ is trivially a model category, so the task of constructing a model structure on $\vcat_S$ is reduced to the problem of constructing a model structure on a category of monoids. 
Thus we obtain:
\begin{theorem}
	Let $\V$ be a combinatorial model category with cofibrant unit, satisfying the monoid axiom.
	Then $\vcat_S$ has a model structure in which a functor $f: \C \to \D$ is:
	\begin{itemize}
		\item A weak equivalence if it is homotopically fully faithful (equivalently, if it is a Dwyer-Kan equivalence),
		\item A fibration if it is a local fibration.
	\end{itemize}
\end{theorem} 
From the description of the weak equivalences and fibrations, we immediately obtain:
\begin{theorem}\label{thm:enriched-adjunction-fixed-object-set}
	Let
	\adjunctiondiagram{L}{\V}{\W}{R}
	be a Quillen adjunction between monoidal model categories, such that the Dwyer-Kan model structures on $\vcat$ and $\wcat$ exist, and the right adjoint $R$ is lax monoidal. Then the induced adjunction
	\adjunctiondiagram{\lcat_S}{\vcat_S}{\wcat_S}{\rcat_S}
	is a Quillen adjunction.
\end{theorem}
\begin{proof}
	$\rcat_S$ preserves the fibrations and acyclic fibrations, since it is defined by applying $R$ locally.
\end{proof}
Thus, there is not too much difficulty provided that we restrict to fixed object sets. The main difficulty arises when trying to glue these various adjunctions together, and from a technical point of view the main challenge is to verify that $\rcat: \wcat \to \vcat$ preserves fibrations; or equivalently that $\lcat: \vcat \to \wcat$ preserves acyclic cofibrations.

\section{Change of enrichment Theorem \ref{alphathm:simple}}\label{section:thmA}
We are now ready to go about stating and proving the simpler of our two main theorems.
\noindent We recall the notion of an interval from \cite{BergerMoerdijk2012}:
\begin{definition}[{\cite[Definition 1.11]{BergerMoerdijk2012}}]\label{dfn:interval-berger-moerdijk}
	Let $\V$ be a monoidal model category with unit $\iv$. 
	Denote by $\I_\V$ the $\V$-category with object set $\{0, 1\}$, and $\I_\V[i,j] = \iv$ for $i,j \in \{0,1\}$, with composition maps given by the unital maps for $\iv$. 
	A \emph{$\V$-interval} is an element of $\vcat_{\{0,1\}}$, weakly equivalent to $\I_\V$. 
\end{definition}
The main point of $\V$-intervals is that Berger and Moerdijk use them to construct the Dwyer-Kan model structure\footnote{
	We should note that Berger and Moerdijk do not approach the model structure on $\vcat$ as being the Dwyer-Kan model structure; instead, they construct the \textit{canonical} model structure, which is specified as being the unique (if it exists) model structure with the acyclic fibrations as we have, and where every object is fibrant. 
	However, they then prove that under their hypotheses, the canonical model structure is also Dwyer-Kan, so for our purposes this distinction does not matter.
}, characterizing the fibrations as the local fibrations which have the RLP against the inclusions $\{0\} \to \H$, for $\H$ a $\V$-interval (recall that $\{0\}$ is the $\V$-category with a single object $0$, and unique hom-object given by $\iv$).
\begin{lemma}\label{lem:single-interval-fibrant-objects}
	Let $\V$ be a monoidal model category satisfying the hypotheses of \cite[Theorem 1.10]{BergerMoerdijk2012}.
	Let $\C$ and $\D$ be fibrant objects of $\vcat$. A map $\alpha: \C \to \D$ of $\vcat$ is a fibration iff: 
	\begin{itemize}
		\item $\alpha$ is a local fibration, and
		\item there exists some $\V$-interval $\mathbb{G}$ such that $\alpha$ has the RLP against the inclusion $\{0\} \hookrightarrow \mathbb{G}$. 
	\end{itemize}
\end{lemma}
\begin{proof}
	By the construction of the model structure on $\vcat$, it suffices to verify that $\alpha$ has the LLP against any inclusion $\{0\} \hookrightarrow \H$, for $\H$ a $\V$-interval. 
	By definition of $\V$-interval, there is a weak equivalence $\mathbb{G} \to \I_f$, where $\I_f$ is a fibrant replacement for $\I$. Factor this as an acyclic cofibration followed by an acyclic fibration in $\vcat_{\{0,1\}}$, to get maps 
	\begin{tikzcd}
		\mathbb{G} \ar[r, "\sim", hook] & \widetilde{\mathbb{G}} \ar[r, "\sim", two heads] & \I_f.
	\end{tikzcd}
	Now, let $\H$ be an arbitary $\V$-interval, equipped with a morphism $\H \to \D$.
	Lifting the map $\H \to \I_f$, we obtain a weak equivalence $\H \to \tilde{\mathbb{G}}$. 
	Apply Brown's Lemma in the model category $\vcat_{\{0,1\}}$ to factor this as 
	\begin{tikzcd}
		\H \ar[r, "\sim", hook] & \K \ar[r, "\sim"] & \widetilde{\mathbb{G}}
	\end{tikzcd}
	Where the second map is left inverse to an acylic cofibration 
	\begin{tikzcd}
		\widetilde{\mathbb{G}} \ar[r, "\sim", hook] & \K
	\end{tikzcd}
	Putting this together, we obtain a commutative diagram
	\begin{center}
	\begin{tikzcd}
		\{0\} 
			\ar[dd, hook, "\sim" labl] 
			\ar[rrrr]
			\ar[rrd]
		&&&&
		\C \ar[dd]
		\\
		& &
		\mathbb{G} 
			\ar[rru, dashed, "g"]
			\ar[d, hook, "\sim" labl]
		\\
		\H 
			\ar[rr, hook, "\sim"]
			\ar[rrrr, shift right = 2,bend right = 10]
		&&
		\K
			\ar[rr, dashed, "f"]
			\ar[rruu, dashed, "h"]
		&&
		\D
	\end{tikzcd}
	\end{center}
	In which:
	\begin{itemize}
		\item $f$ exists by applying the fact that $\D$ is fibrant and $\H \to \K$ is an acylic cofibration.
		\item $g$ exists by the assumption that $\alpha: \C \to \D$ has the RLP against the map $\{0\} \to \mathbb{G}$.
		\item $h$ exists by pulling back the diagram to $\vcat_{\{0,1\}}$ and applying the fact that $\mathbb{G} \to \K$ is an acyclic cofibration, and that fibrations in $\vcat_{\{0,1\}}$ are exactly the local fibrations.
	\end{itemize}
	Then, the composite 
	\begin{tikzcd}
		\H \ar[r, hook, "\sim"] & \K \ar[r, dashed] & \C
	\end{tikzcd}
	gives a lift. 
	We conclude that $\alpha: \C \to \D$ has the RLP against $\V$-intervals, and (as it is a local fibration by assumption) it is therefore a fibration.
\end{proof}

\begin{lemma}\label{lem:fibrations-are-naive}
	Let $\V$ be a monoidal model category satisfying the hypotheses of \cite[Theorem 1.10]{BergerMoerdijk2012}. 
	Then every fibration in $\vcat$ is in particular a naive fibration.
\end{lemma}
\begin{proof}
	Let $f: \C \to \D$ be a fibration in $\vcat$. Then $f$ is by definition a local fibration, and so it suffices to show that $\pi_0 f$ is an isofibration. 
	\par Since $\V$ is in particular \textit{adequate} \cite[Definition 1.1]{BergerMoerdijk2012} it satisfies the \textit{coherence axiom} \cite[Definition 2.19]{BergerMoerdijk2012} by \cite[Lemma 2.23]{BergerMoerdijk2012}. 
	The coherence axiom implies that any commutative square in $\cat$
	\begin{center}
	\begin{tikzcd}
		\{0\} \ar[r] \ar[d]
		&
		\pi_0 \C \ar[d, "\pi_0 f"]
		\\
		\I_\cat \ar[r] 
		& \pi_0\D
	\end{tikzcd}
	\end{center}
	(where $\I_\cat$ is the walking isomorphism) arises by applying $\pi_0$ to a commutative square in $\vcat$ of the form
	\begin{center}
	\begin{tikzcd}
		\{0\} \ar[r] \ar[d]
		&
		\C \ar[d, "f"]
		\\
		\H \ar[r] 
		& \D
	\end{tikzcd}
	\end{center}
	with $\H$ a $\V$-interval. 
	Since $f$ is a fibration, this diagram in $\vcat$ has a lift; applying $\pi_0$ to this lift gives a lift for the former diagram in $\cat$, hence showing that $\pi_0 f$ is an isofibration.
\end{proof}

\begin{theorem}[Theorem \ref{alphathm:simple}]\label{thm:naive-fibration-transfer}
	Let $\W$ be a monoidal model category satisfying the conditions of \cite[Theorem 1.10]{BergerMoerdijk2012}. 
	Let $L: \V \rightleftarrows \W: R$ be a Quillen adjunction with $R$ lax monoidal, and suppose that: 
	\begin{enumerate}
		\item The Dwyer-Kan model structure on $\vcat$ exists and has the naive fibration property (Definition \ref{dfn:naive-fibration-property}).
		\item $\rcat$ preserves naive fibrations between fibrant objects.
	\end{enumerate}
	Then the induced adjunction 
		\adjunctiondiagram{\catfunctor{L}}{\vcat}{\wcat}{\catfunctor{R}}
	is a Quillen adjunction, and $\wcat$ has the naive fibration property. 
\end{theorem}
\begin{proof}
	By \cite[Proposition 6.18]{Joyal2007}, it suffices to show that $\rcat$ preserves acyclic fibrations, as well as fibrations between fibrant objects. 
	The first is evident, as acyclic fibrations are defined locally and $\rcat$ is defined by applying $R$ (which preserves acyclic fibrations) locally. So, let $f : \C \to \D$ be a fibration in $\wcat$ between fibrant objects. 
	By Lemma \ref{lem:fibrations-are-naive}, $f$ is a naive fibration in $\wcat$. 
	Then $\rcat(f)$ is (by condition (2)) a naive fibration in $\vcat$, and hence has the RLP with respect to the intervals in $\vcat$. 
	It follows in particular by condition (1) that $f$ is a fibration in $\vcat$, and hence the induced adjunction is Quillen. 
	\par 
	Now we show that $\wcat$ has the naive fibration property: let $f: \C \to \D$ be a naive fibration between fibrant objects of $\wcat$. 
	Then $\rcat(f)$ is a naive fibration, so it has the RLP against $\{0\} \to \G$ for some $\V$-interval $\G$. 
	Hence, $f$ has the RLP against $\lcat(\{0\} \to \G)$.
	Note that $\lcat(\{0\}) = \{0\}$, since $\lcat$ is a left adjoint when restricted to $\vcat_{\{0\}}$, and $\{0\}$ is the initial object of $\vcat_{\{0\}}$. 
	Since $\lcat$ is left Quillen, $\{0\} \cong \lcat(\{0\}) \to \lcat\G$ is an acyclic cofibration, in particular $\G$ is an interval. Hence by Lemma \ref{lem:single-interval-fibrant-objects}, $f$ is a fibration, as desired.
\end{proof}
Hypothesis (2) of Theorem \ref{thm:naive-fibration-transfer} is phrased in a somewhat general form. 
The following gives an easy to verify (and commonplace) condition which implies that this hypothesis holds.
\begin{proposition}
	Suppose in the setting of Theorem \ref{thm:naive-fibration-transfer} that the colax structure map $L(\iv) \to \iw$ has a left inverse in $\Ho(\W)$; i.e. such that it makes $L(\iv)$ a retract of $\iw$ in $\Ho(\W)$. Then condition (2) is satisfied. In particular, if $L(\iv) \to \iw$ is a weak equivalence, then condition (2) is satisfied.
\end{proposition}
\begin{proof}
	Since $\rcat$ preserves local fibrations, it remains only to show that it preserves isofibrations of homotopy categories. Let $f: \C \to \D$ be a morphism in $\wcat$ which induces an isofibration $\pi_0 f : \pi_0 \C \to \pi_0\D$; we wish to verify that $\pi_0 \rcat f$ is an isofibration of categories. Let $\iota: L(\iv) \to \iw$ be the colax structure map in $\Ho(\W)$, and $r: \iw \to L(\iw)$ its right inverse in $\Ho(\W)$.
	We have a retract diagram
	\begin{center}
	\begin{tikzcd}[sep = huge]
		\pi_0\rcat\C 
			\ar[d, "\pi_0 \rcat f"] 
			\ar[r, "r^*"]
		& \pi_0 \C 
			\ar[d, "\pi_0 f"]
			\ar[r, "\iota^*"]
		& \pi_0 \rcat \C
			\ar[d, "\pi_0 \rcat f"] 
		\\
		\pi_0\rcat\D
			\ar[r, "r^*"]
		& \pi_0 \D
			\ar[r, "\iota^*"]
		& \pi_0 \rcat \D
	\end{tikzcd}
	\end{center}
	in which the vertical arrows are induced by $f$ and the horizontal arrows are induced by $\iota$ and $r$, using the fact that 
	$$\pi_0 \rcat \C [x,y] = \Ho(\V)[\iv, \rcat(\C)[x,y]] \cong \Ho(\V)[\iv, R(\C[x,y])] \cong \Ho(\W)[L(\iv), \C[x,y]]$$
	Since isofibrations are closed under retracts, the result now follows.
\end{proof}
We close this section by demonstrating how Theorem \ref{thm:naive-fibration-transfer} gives a general method to transfer proofs of the naive fibration property to other settings.
\begin{corollary}\label{thm:naive-fibration-cartesian}
	Let $\V$ be a cartesian monoidal model category satisfying the hypotheses of \cite[Theorem 1.10]{BergerMoerdijk2012}. Then $\vcat$ has the naive fibration property.
\end{corollary}
\begin{proof}
	Let $L: \sset \to \V$ be a left Quillen functor which sends the terminal object of $\sset$ to a cofibrant replacement of the terminal object of $\V$. Such a functor exists (and is ``unique up to contractible choice'') by e.g. the universal property of the model category $\sset$ \cite{Dugger2001}. 
	Then applying Theorem \ref{thm:naive-fibration-transfer} to this adjunction gives the desired result.
\end{proof}
Of course, the argument of the above theorem applies to more than simply cartesian monoidal categories; it applies to any monoidal model category equipped with a lax monoidal right adjoint $\V \to \sset$ whose left adjoint weakly preserves the unit. 
For instance, this recovers the fact that $\dg-\cat$ has the naive fibration property.
One can compare Higher Topos Theory Theorem A.3.2.24, which gives a model structure on $\vcat$ which satisfies the naive fibration property provided $\V$ is ``excellent'' \cite{LurieHTT}.
We suspect one could enhance the above theorem by mapping not out of $\sset$, but out of some ``universal'' monoidal model category $\V$, to show that all $\V$ satisfying Berger and Moerdijk's hypotheses (perhaps with some additional assumptions) have the naive fibration property.
We will not attempt such a thing here, however.

\section{Groupoid-equivariant chain complexes}
\label{section:chgpd}
Henceforth all chain complexes are assumed to be over $\mathbb{Z}$, and are homologically graded, concentrated in nonnegative degrees. 
We denote the category of such chain complexes $\chbasic$. 
We denote the by $\otimes$ the standard tensor product of chain complexes, and by $\mathbb{Z}[0]$ the unit for this monoidal product.
\par 
Denote by $\dgalg$ the category of $\dg$-algebras; i.e. monoids in $\chbasic$ with respect to $\otimes$. Similarly denoteby $\dgalgaug$ the category of \textit{augmented} $\dg$-algebras; i.e. chain complexes $A_\bullet$ equipped with the structure of a monoids for $\otimes$ and also equipped with a map of monoids $A_\bullet \to \mathbb{Z}[0]$, where $\mathbb{Z}[0]$ is the unit for $\otimes$. 
Then there is an adjunction
\adjunctiondiagram{I/I^2}{\dgalgaug}{\chbasic}{-\oplus \mathbb{Z}[0]}
called the ``square zero/indecomposables'' adjunction. 
The right adjoint takes a chain complex $C$ to the $\dg$-algebra $C \oplus \mathbb{Z}[0]$, with multiplication defined by imposing that elements of $C$ multiply together to $0$, and that the natural inclusion $\mathbb{Z}[0] \hookrightarrow C \oplus \mathbb{Z}[0]$ is the identity structure map for the monoid object $C \oplus \mathbb{Z}[0]$. 
The left adjoint $I/I^2$ takes a $\dg$-algebra with augmentation map $\epsilon: A \to \mathbb{Z}[0]$ to the chain complex $\ker(\epsilon)/\ker(\epsilon)^2$. 
\par 
The derived functor of $I/I^2$ is computed by the bar complex of an augmented $\dg$-algebra, and coincides with the Quillen homology \cite{Harper2008}.
For our purposes, there is another point of view we would like to take on constructing this adjunction: there is an adjunction
\adjunctiondiagram{U}{\chbasicaug}{\chbasic}{-\oplus \mathbb{Z}[0]}
where $U$ forgets the augmentation. The right adjoint $-\oplus \mathbb{Z}[0]$
is lax monoidal, when $\chbasic$ is equipped with its \textit{cartesian} monoidal product $\oplus$ and $\chbasicaug$ is equipped with the standard monoidal product $\otimes$. 
Hence by the arguments in Section \ref{section:background}, it induces an adjunction
\adjunctiondiagram{U^{\mathrm{Mon}}}{\mathrm{Mon}_{\otimes}(\chbasicaug)}{\mathrm{Mon}_{\oplus}(\chbasic)}{(- \oplus \mathbb{Z}[0])^{\mathrm{Mon}}}
between categories of monoids, in which the right adjoint is given by applying $-\oplus\mathbb{Z}[0]$ to underlying objects, but the left adjoint is not given by simply applying $U$ to underlying objects. 
Every chain complex has a unique monoid structure with respect to $\oplus$, hence there is an equivalence of categories 
$\mathrm{Mon}_{\oplus}(\chbasic) \simeq \chbasic$, and this recovers the indecomposables/square zero adjunction above.
\par 
For the next few sections, we will use the change-of-enrichment point of view to generalize this adjunction to $\dg$-categories. 
Of course, $\chbasic$ is not a \textit{closed} monoidal category with respect to the cartesian product $\oplus$, 
as no nontrivial category with a zero object can be cartesian closed. 
To rectify this, we shall work with chain complexes over varying groupoids. 
\subsection{Groupoid-equivariant chain complexes}
We recall the notion of a chain complex with a groupoid of operators, and the resulting category $\chgpd$. Our primary reference is \cite{Brown_Higgins_Sivera_2011}. 
\par 
Let $G$ be a groupoid. A \emph{$G$-chain complex} is a functor $G \to \chbasic$. 
Denote the resulting functor category by $[G, \chbasic]$.
The assignment $G \mapsto [G, \chbasic]$ gives a functor $\gpd^{op} \to \cat$, and we define the category of groupoid-equivariant chain complexes $\chgpd$ to be the Grothendieck construction
$$\chgpd = \int_{\gpd}[-, \chbasic]$$
Thus, an object of $\chgpd$ is a pair $(G, C_\bullet)$ where $C_\bullet$ is a functor $G \to \chbasic$, and a morphism $(G, C_\bullet) \to (H, D_\bullet)$ consists of a morphism $F: G \to H$ of groupoids along with a natural transformation $C_\bullet \Rightarrow D_\bullet \circ F$. 
For a fixed groupoid $G$, can identify $[G,\chbasic]$ as equivalent to a disjoint union of categories $\coprod_{i \in I} \ch_{\ge 0}\left(\Z[G_i]\right)$, where $I$ ranges over the connected components of $G$ and $G_i$ is a group equivalent (as a groupoid) to the $i$th component of $G$. 
Thus, each of these admits the standard projective model structure. 
\begin{theorem}
 	The category $\chgpd$ admits a cofibrantly generated model structure such that a morphism $(f, \eta): (G, C_\bullet) \to (H, D_\bullet)$ is:
 	\begin{itemize}
 		\item A weak equivalence if $f$ is a weak equivalence of groupoids and $\eta: C_\bullet \to f^*D_\bullet$ is a weak equivalence of $G$-equivariant chain complexes,
 		\item A fibration if $f$ is a fibration of groupoids (i.e. an isofibration) and $\eta$ is a fibration (i.e. a surjection) of $G$-equivariant chain complexes
 		\item A cofibration if it has the left lefting property against the acyclic fibrations.
 	\end{itemize}
 	The generating (acylic) cofibrations are the generating (acylic) cofibrations of $\gpd$ (viewed as equivariant chain complexes with constant $\zero$ chain complex) along with the generating (acyclic) cofibrations of $\chbasic$ (viewed as equivariant chain complexes over the terminal groupoid).
\end{theorem}
\begin{proof}
	The existence of the model structure follows from the main result of \cite[Theorem 3.0.12]{Harpaz2014}. The statement on cofibrant generation is easy to check.
\end{proof}
The category $\chgpd$ admits two monoidal structure of interest to us: first, a tensor product $\otimes$, combining the cartesian product of groupoids with the tensor product of chain complexes, induced by the description of $\chgpd$ as a Grothendieck construction. 
We will refer to this as the ``tensor product'' of chain complexes.
The second monoidal product is the cartesian (categorical) product $\times$. 

\subsection{The monoidal (model) structures on $\chgpd$}
\begin{definition}\label{dfn:tensor-product-on-chgpd}
	We denote by $\otimes$ the monoidal product on $\chgpd$ induced by the tensor product $\otimes_G$ of $G$-equivariant chain complexes.
	Explicitly, for $(G,C_\bullet)$ and $(H, D_\bullet)$ in $\chgpd$, we have 
	$$(G,C_\bullet) \otimes (H, D_\bullet) := (G \times H, \pi_G^*(X) \otimes_{G \times H} \pi_H^*(Y))$$
	where $\pi_G$ and $\pi_H$ are the projection maps from $G \times H$, 
	and $\otimes_{G \times H}$ is the tensor product in the category of chain complexes over the groupoid $G \times H$
	(which can be defined by viewing this category as equivalent to a disjoint union of categories of chain complexes over group rings).
	We will refer to this as the \emph{tensor product} for $\chgpd$.
\end{definition}
For more detail on this tensor product, see \cite[Chapter 9]{Brown_Higgins_Sivera_2011}.

\begin{lemma}\label{lem:chgpd-pushout-product}
	$\chgpd$ satisfies the pushout product axiom with respect to both the tensor product $\otimes$ and the cartesian product $\times$.
\end{lemma}
\begin{proof}
	By \cite[Lemma 2.3]{SchwedeShipley1998} it suffices to check on the (acylic) cofibrant generators, which is straightforward.
\end{proof}

\begin{theorem}\label{lem:chgpd-monoid-axiom}
	$\chgpd$ satisfies the monoid axiom for both the tensor product $\otimes$ and the cartesian product $\times$. 
\end{theorem}
\begin{proof}
	By \cite[Lemma 2.3]{SchwedeShipley1998} it suffices to check on the generating acyclic cofibration. From the above theorem, there are two sorts of generating acyclic cofibrations: the ``groupoid'' one, the inclusion $\bullet \hookrightarrow \I$ where $\I$ is the walking isomorphism, and the ``chain complex'' ones, the inclusions $\mathbf{0} \hookrightarrow D^n$ where $D^n$ is the chain complex (over the terminal groupoid) with two generators $a$ and $b$ in dimensions $n$ and $(n-1)$, with $d(a) = b$. 
	\par 
	For the cartesian product, it is straightforward to see that the product of a generating acylic cofibration with any object is an acyclic cofibration, from which the result follows. 
	\par 
	For the product $\otimes$, the product of $\bullet \hookrightarrow \I$ with any object is an acylic cofibration. So it remains to check the generators $\mathbf{0} \hookrightarrow D^n$. 
	Note that weak equivalences in $\chgpd$ are closed under transfinite composition. 
	Hence, it suffices to check that a pushout along the tensor product $X \otimes (\mathbf{0} \hookrightarrow D^n)$ is always a weak equivalence. 
	This now follows from the fact that the category of chain complexes over any ring (in particular, any group ring $\mathbb{Z}[G]$) satisfies the monoid axiom \cite[Section 4]{SchwedeShipley1998}.
\end{proof}

\begin{theorem}\label{thm:chgpd-monoidal-model-for-tensor}
	$\chgpd$ is a monoidal model category with respect to the tensor product $\otimes$, satisfying the monoid axiom. 
\end{theorem}
\begin{proof}
	Closure of the monoidal structure is shown in \cite[Chapter 9]{Brown_Higgins_Sivera_2011}. 
	Since the unit is cofibrant, the fact that it satisfies the pushout product axiom (Lemma \ref{lem:chgpd-pushout-product}) implies that it satisfies the unit axiom. The monoid axiom is Lemma \ref{lem:chgpd-monoid-axiom}).
\end{proof}
We will actually be slightly more interested in \textit{augmented} chain complexes equipped with their tensor product.
\begin{definition}\label{dfn:augmented-chgpd}
	Let $\mathbb{Z}[0]_\bullet \in \chgpd$ be the unit for $\otimes$ (i.e. the object with groupoid part the terminal groupoid, and chain complex part $\mathbb{Z}$ concentrated in degree $0$).
	We denote by $\chaug$ the slice (model) category of $\chgpd$ over $\mathbb{Z}[0]_\bullet$.
	As the slice category over a monoid, $\chaug$ inherits a tensor monoidal structure from $\chgpd$, which we shall also denote $\otimes$. For $f: X \to \I$ and $g: Y \to \I$, their tensor product is given by the composition
	$$f \otimes g: X  \otimes Y \to \I \otimes \I \to \I$$
	where the second map is given by the unital structure of $\I$. 
\end{definition}
\begin{theorem}\label{thm:chaug-monoidal-model-for-tensor}
	$\chaug$ is a monoidal model category with respect to the tensor product $\otimes$, and satisfies the monoid axiom.
\end{theorem}
\begin{proof}
	Since colimits in slice categories are calculated on underlying objects, $\otimes$ preserves colimits in $\chaug$ and is therefore a left adjoint; i.e. $\otimes$ is closed.
	The other axioms follow for similar reasons, and because the model structure on $\chaug$ has cofibrations/fibrations/weak equivalences given by commutative triangles whose underlying morphism is a cofibration/fibration/weak equivalence in $\chgpd$.
\end{proof}
We now move on to the cartesian monoidal structure on $\chgpd$. 
\begin{definition}
	We denote by $\times$ the cartesian product on $\chgpd$, which is given by
	$$(G,C_\bullet) \times (H, D_\bullet) := (G \times H, \pi_G^*(X) \oplus_{G \times H} \pi_H^*(Y))$$
\end{definition}
It is not immediately obvious that $\chgpd$ is cartesian closed: indeed, unlike the situation with $\otimes$, the product 
that we are inducing from is not closed, which follows immediately from the fact that for a fixed groupoid $G$, the category $\chbasic^G$ has a zero object.
However, it follows from \cite[Theorem 9.11, Example 9.19]{Nunes2024} that for any category $\A$ with finite biproducts, the category of groupoid-equivariant objects of $\A$ is cartesian closed.
As their results are phrased rather type theoretically, we will construct the cartesian closed structure here without reference to the type theoretic language. 
We will not actually use the explicit description of the closed structure, but we include it for completeness. 
We note that we will define a cartesian closed structure on $\chgpd$ in particular, but our construction will work to construct a cartesian closed structure on groupoid-equivariant objects in any category $\A$ with finite biproducts. 
\begin{definition}\label{dfn:cartesian-mapping-groupoid-chain-complex}
	For $(G,A),(H,B) \in \chgpd$, the groupoid chain complex $[(G,A),(H,B)]_\times = (K,C)$ is the chain complex with:
	\begin{enumerate}
		\item $K_0 = \Hom_{\chgpd}((G,A),(H,B))$
		\item
		For $f,g \in K_0$, the hom-set $K(f,g)$ is the set of set-functions $\eta: G_0 \to H_1$ such that:
		\begin{itemize}
			\item For any morphism $\ell: x \to y$ in $H_1$, the diagram
			\begin{center}
			\begin{tikzcd}
				f(x) \ar[r, "f(\ell)"] \ar[d, "\eta(x)"] & f(y) \ar[d, "\eta(y)"] \\
				g(x) \ar[r, "g(\ell)"] & g(y)
			\end{tikzcd}
			\end{center}
			commutes. 
			\item For any $x \in G_0$, and $\alpha \in A(x)$, 
			$$g(\alpha) = f(\alpha)^{\eta(x)}$$
		\end{itemize}
		\item For $f \in K_0$, the chain complex $C(f)$ is the limit of the composite diagram $f: G \to H \to \chbasic$. More concretely, the $n$-chains $C(f)_n$ are given by
		$$\{(\alpha_g)\in \prod_{g \in G_0} B(f(g)) \; | \; \ell(\alpha_g) = \alpha_{\ell(g)} \; \forall \ell \in G_1 \}$$
		and the differential is given componentwise.
		\item For $f,g \in K_0$ and $\eta \in K(f,g)$, the morphism $C(\eta): C(f) \to C(g)$ is given by applying limits to the natural transformation defined by $\eta$. 
	\end{enumerate}
	The last part of the definition makes sense because the two conditions on the $\eta \in K(f,g)$ implies that $\eta$ does indeed define a natural transformation between the composite diagrams $f,g: G \to H \to \chbasic$. 
\end{definition}

\begin{example}
	Let $G$ and $H$ be groups. Then for a $G$-chain complex $A_\bullet$ and an $H$-chain complex $B_\bullet$, the exponential $[(G,A),(H,B)]_\times = (K,C)$ can be described as:
	\begin{itemize}
		\item $K_0 = \Hom_{\chgpd}((G,A_\bullet),(H,B_\bullet))$; that is, an object of the groupoid $K$ is a group homomorphism $f: G \to H$ along with a $G$-chain complex homomorphism $A_\bullet \to f^*B_\bullet$.
		\item For $f,g \in K_0$, the hom-set $K(f,g)$ is the set of $h \in H$ which conjugate $f$ to $g$
		\item For $f \in K_0$, the chain complex $C(f)$ is the subcomplex of fixed points of the image of $G$ (under $f$). 
	\end{itemize}
	It is not too hard to verify that this gives an exponential object for the subcategory of $\chgpd$ where the groupoid part has a single object, i.e. is a group.
\end{example}
\begin{theorem}
		$[(G,A), (H,B)]_\times$ as defined above is the exponential for $\chgpd$, i.e. gives it the structure of a cartesian closed category. 
\end{theorem}
\begin{proof}
	The fact that $\chgpd$ is cartesian closed follows from \cite[Theorem 9.19, Example 9.11]{Nunes2024}. 
	We know that $[-,-]_\times$ is unique up to isomorphism. By considering maps $\mathbb{Z}[n] \to[(G,A), (H,B)]_\times$ and using the adjunction, one obtains the desired isomorphism.
\end{proof}

\begin{theorem}\label{thm:chgpd-monoidal-model-for-cartesian}
	$\chgpd$ is a cartesian monoidal model category, satisfying the monoid axiom.
\end{theorem}
\begin{proof}
	We have just noted that it is cartesian closed. The other facts follow from Lemmas \ref{lem:chgpd-pushout-product}
	and \ref{lem:chgpd-monoid-axiom}.
\end{proof}

\section{Track $\dg$-categories and track chain-categories}
\label{section:chgpd-categories}

\subsection{Track $\dg$-categories}
Let $\chgpd^\otimes$ denote the monoidal model category of groupoid-equivariant chain complexes, with the monoidal product induced by the tensor product on chain complexes. 
In the last section we proved that this is a cofibrantly generated monoidal model category, and satisfies the monoid axiom.
It is easy to check that the unit is cofibrant, and hence the results of \cite{Muro2012} give us a model structure on the category $\trackdgcat$ of categories enriched in this. 
\begin{definition}
	A \emph{track $\dg$-category} is a category enriched in $\chgpd$ with respect to its tensor monoidal structure. We denote the resulting (model) category by $\trackdgcat$.
\end{definition}
\begin{remark}
	As an example, every $\dg$-category clearly gives a track $\dg$-category, where every groupoid is the trivial groupoid. 
Note, though, that these objects will never be cofibrant in the Dwyer-Kan model structure on track $\dg$-categories (except in the case that they have only one object). 
	We think it plausible that one could fix the underlying track category and construct a model structure on the category of track $\dg$-categories with fixed underlying track-category, and then view $\trackdgcat$ as a Grothendieck construction. We will not pursue this line of inquiry here, however. 
\end{remark}
$\dg$-categories arise naturally as a setting in which to do homological algebra.
Track $\dg$-categories should arise as a setting in which to do homological algebra of local systems (with varying groupoid). 
The example which we are most interested in is the track $\dg$-category of spaces, in which an object is a space, and $\Hom(X,Y)$ has:
\begin{itemize}
	\item groupoid part given by the fundamental groupoid of the mapping space from $X$ to $Y$, and
	\item chain complex given by the local system of chains on the \textit{universal cover of} the mapping space from $X$ to $Y$. 
\end{itemize}
In \cite{Strong2025}, we describe a functor which takes an $(\infty,1)$ category to a track $\dg$-category,
refining the functor which takes a $(\infty,1)$-category to a $\dg$-category. 
\par 
In order to obtain the adjunction that will recover the bar construction of a $\dg$-category, we will need an augmented version:
\begin{definition}
	An \emph{augmented track $\dg$-category} is a category enriched in $\chaug$ with respect to its tensor monoidal structure. Equivalently, it is an object of the slice category of $\trackdgcat$ over the category with one object, with mapping object the unit for the tensor monoidal structure on $\chgpd$. 
	We denote the resulting category by $\trackdgcataug$. 
\end{definition}

\subsection{Track $\ch$-categories}
As noted in the introduction, it is uninteresting to enrich in the cartesian monoidal category of chain complexes, because $\chbasic$ has finite biproducts.
Hence, the cartesian monoidal structure is the same as the cocartesian monoidal structure, and therefore categories enriched with respect to this monoidal structure are essentially just chain complexes, by the following theorem:
\begin{theorem}\label{thm:cocartesian-enrichment-trivial}
	Let $\V$ be a cocartesian monoidal category, and $\C$ a $\V$-category. Then all hom-objects of $\C$ are canonically isomorphic via the composition operation.
\end{theorem}
This theorem is certainly well known, but lacking a reference we shall prove it here.
\begin{proof}
	First, note that the (left) identity axiom says that there is a map $\id: \emptyset \to \C[x,x]$, such that for any $y$, the composition
	\begin{center}
	\begin{tikzcd}[sep=large]
		\C[x,y] \ar[r, "\sim "] 
			& \emptyset \coprod \C[x,y] \ar[r, "\id \coprod {\C[x,y]}"] 
			& \C[x,x] \coprod  \C[x,y] \ar[r, "\circ "] 
			& \C[x,y]
	\end{tikzcd}
	\end{center}
	is the identity. Of course, the map $\id$ is immaterial to unwinding this statement: It merely says that the component $\C[x,y] \to \C[x,y]$ of the composition map $\C[x,x] \coprod \C[x,y] \to \C[x,y]$ is the identity. 
	Now, let $x,y,a,b \in \ob \C$. We have a map
	\begin{center}
	\begin{tikzcd}
		\C[x,y] \ar[r, "\iota "] 
			& \C[a,x] \coprod \C[x,y] \coprod \C[y,b] \ar[r, "\circ "]
			& \C[a,b]
	\end{tikzcd}
	\end{center} 
	Denote this map by $f^{xy}_{ab}: \C[x,y] \to \C[a,b]$. 
	Applying associativity tells us that the square 
	\begin{center}
	\begin{tikzcd}[sep = huge]
		\C[x,a] \coprod \C[a,x] \coprod \C[x,y] \coprod \C[y,b] \coprod \C[b,y] 
			\ar[r] \ar[d, "{\C[x,a] \coprod {f^{xy}_{ab}} \coprod \C[b,y]}"]
		& \C[x,x] \coprod \C[x,y] \coprod \C[y,y] 
			\ar[d] 
		\\
		\C[x,a] \coprod \C[a,b] \coprod \C[b,y] 
			\ar[r, "{f^{ab}_{xy}}"] 
		& \C[x,y]
	\end{tikzcd}
	\end{center}
	commutes. Thus, $f^{xy}_{ab}: \C[x,y] \to \C[a,b]$ and $f^{ab}_{xy}: \C[a,b] \to \C[x,y]$ are inverse. 
	From this and associativity we see that the data of all the composition maps is exactly the data of (a contractible groupoid of) isomorphisms between the hom objects.
\end{proof}
Unlike $\chbasic$, $\chgpd$ does \textit{not} have biproducts, and so enriching in $\chgpd$ can be nontrivial. Indeed, $\chgpd$-categories we should expect to be at least as interesting as $\gpd$-categories. 
\begin{definition}
	A \emph{track chain-category} is a category enriched in $\chgpd$ with respect to its cartesian monoidal structure. We denote the resulting (model) category by $\trackchaincat$.
\end{definition}
To understand generally what a $\chgpd^\times$-category is, we first consider the simplest case: groupoid with no nonidentity morphisms. 
\par
Let $\C \in \trackchaincat$ be such that for any $x,y \in \ob \C$, the groupoid $\Pi_1 \C[x,y]$ has no nonidentity morphisms. 
Then, the basic data of $\C$ consists of an ordinary category $C$, together with, for every morphism $f \in \mor \pi_0C$ a chain complex $\C(f)$. 
The composition operations in $\C$ give us, for every composible $f$ and $g$ in $\mor \pi_0C$, a map $\C(f) \oplus \C(g) \to \C(f g)$. 
Since $\oplus$ is a coproduct, this is equivalent to maps $\C(f) \to \C(f g)$ and $\C(g) \to \C(fg)$. 
So, the data of the composition operation of $\C$ is the data of the composition operation of $\C$, along with, for every composable $f$ and $g$, maps $\rightarrow_g: \C(f) \to \C(fg)$ and $\rightarrow^f: \C(g) \to \C(fg)$.
Thus, we may imagine this to be a system of chain complexes parameterized on the \textit{morphisms} of the underlying track category.
More generally, if $\C \in \trackchaincat$ is free on a collection of local $0$-cells (morphisms), with no local $1$-cells (2-morphisms), then the data of $\C$ is the same as a functor from the poset of morphisms in $\C$ (where $f \le g$ if there exists $h$ such that $f \circ h = g$ or $h \circ f = g$) to the category of chain complexes.
\begin{example}
	Let $\C \in \trackchaincat$ have one object, with underlying track category the ordinary $1$-category with single  idempotent nonidentity morphism $f$.
	Then the data of the underlying graph of $\C$ consists of two chain complexes, $A_1$ and $A_f$. 
The composition operation endows these with the following collection of potentially nonidentity maps:
	\begin{itemize}
		\item $1_\ell: A_1 \to A_f$
		\item $1_r: A_1 \to A_f$
		\item $f_r: A_f \to A_f $
		\item $f_\ell: A_f \to A_f$
	\end{itemize}
	Asociativity and unitality constraints impose that these must satisfy the following compatibility conditions:
	$$1_\ell = f_\ell \circ 1_\ell \quad 1_r = f_r \circ 1_r \quad f_\ell^2 = f_\ell \quad f_r^2 = f_r \quad f_\ell \circ f_r = f_r \circ f_\ell \quad f_r \circ 1_\ell = f_\ell \circ 1_r$$
\end{example}
\subsection{The induced adjunction}
Consider the adjunction 
\adjunctiondiagram{U}{\chaug^\otimes}{\chgpd^\times}{F}
In which the left adjoint $U$ forgets the augmentation, and the right adjoint $F$ is given by 
taking the product with the unit $\I$ (and setting the augmentation map to be the projection). 
Then the right adjoint $F$ is lax monoidal, since $\chaug^\otimes$ is semicartesian (i.e. the unit of $\chaug^\otimes$ is the terminal object). 
We therefore obtain an induced adjunction\footnote{
	The notation $\dgbarconstruction$ is due to the connection of this adjunction to the strictification of $(\infty,1)$-categories, see \cite{Strong2025}.
	Essentially, track $\dg$-categories are very close to what one might call ``gray $(\omega,1)$-categories,'' whereas track chain-categories are close to $(\omega,1)$-categories. 
	The functor $\dgbarconstruction$ is related to ``strictification,'' though it is more clear in this setting that it may be thought of as a sort of ``linearization.''
}
\adjunctiondiagram{\dgbarconstruction}{\trackdgcataug}{\trackchaincat}{\dgbaradjoint}
We will iteratively apply Theorem \ref{thm:naive-fibration-transfer} to deduce that it is a Quillen adjunction, and furthermore that both categories have the naive fibration propery. First, we note the following:

\begin{theorem}\label{thm:track-dg-cat-naive-fibration}
	$\trackdgcataug$ has the naive fibration property.
\end{theorem}
\begin{proof}
	There is a functor from $\sset$ to $\chaug$, given by composing the fundamental crossed complex functor $\pi$ (see e.g. \cite{Tonks_2003}) with (the slice category version of) the functor $\nabla$ of \cite[Section 7.4]{Brown_Higgins_Sivera_2011}. Since $\pi$ is weak Quillen monoidal and $\nabla$ is strong monoidal, the hypotheses of Theorem \ref{thm:naive-fibration-transfer} apply.
\end{proof}
	
\begin{lemma}\label{lem:all-0-cells-isomorphisms}
	Let $\C \in \trackchaincat$. Let $a : \mathbb{Z}[0]_\bullet \to \C[x,y]$ be a morphism of groupoid chain complexes such that the image of the unique object $\bullet$ is an isomorphism in $\pi_0\C$. 
	Then the adjoint to $a$ represents an isomorphism in the homotopy category of $\dgbaradjoint(\C)$.
\end{lemma}
\begin{proof}
	Denote the image of the unique object $\bullet $ by $p$, and the image of the generator of $\mathbb{Z}$ by $\alpha$. 
	The assumption that $p$ is an isomorphism in $\pi_0\C$ means that there is some object $q$ of $\Pi_1\C[y,x]$, along with morphisms
	$\ell_x: 1_x \to pq$ and $\ell_y: 1_y \to qp$. 
	Define the ``whiskering by $pq$'' map to be
	$$\C[a,x](r) \to \C[a,x](rpq) \quad \quad \alpha \mapsto \circ(\alpha, 0_{pq}) = \circ (\alpha, 0_p,0_q)$$
	Then, up to $\ell_x$, this is the same as the map $\alpha \mapsto \circ(\alpha, 0_{1_x}) = \alpha$, so it follows that this is an isomorphism.
	Similarly define a ``whiskering by $qp$" map.
	Using the fact that these are isomorphisms, and associativity of the composition, we can see that the maps $\psi^p_{x}: \C[x,y](p) \to \C[x,x](1_x)$ and $\psi^q_x: \C[y,x](q) \to \C[x,x](1_x)$ are both isomorphisms. 
	Now consider the map $b: \mathbb{Z}[0]_\bullet \to \C[y,x]$ given by $(q, {\psi^q_x}^{-1} \circ \psi^p_x (-\alpha))$. 
	Unwinding the definition of the compositions in $\dgbaradjoint(\C)$, we see that the adjoint to $b$ is a an inverse to $a$ up to the actions of the groupoids $\Pi_1\dgbaradjoint(\C)[x,x]$ and $\Pi_1 \dgbaradjoint(\C)[y,y]$, and hence represents an inverse in the homotopy category. 
\end{proof}

\begin{theorem}\label{thm:dgbaradjoint-preserves-naive-fibrations}
	The functor $\dgbaradjoint: \trackchaincat \to \trackdgcataug$ preserves naive fibrations.
\end{theorem}
\begin{proof}
	Let $f: \C \to \D$ be a naive fibration in $\trackchaincat$. Since $\dgbaradjoint$ clear preserves local fibrations, it suffices to show that $\dgbaradjoint(f)$ induces an isofibration of homotopy categories. 
	Let $a: \mathbb{Z}[0]_\bullet \to \D[x,y]$ descend to an isomorphism in $\pi_0 \dgbaradjoint(\D)$. 
	We wish to lift this isomorphism to $\pi_0\dgbaradjoint(\C)$, given already a lift $\widehat{x}$ for the object $x \in \pi_0 \dgbaradjoint(\C)$.
	Now, the data of $a$ consists of:
	\begin{itemize}
		\item an object $p$ of $\Pi_1(\D[x,y])$
		\item an element $\alpha$ of the group $\D[x,y](p)_0$
	\end{itemize}
	Since $a$ descends to an isomorphism, in particular $p$ must be a homotopy equivalence in the underlying track-category. 
	Since $f$ is a naive fibration, it is a fibration of underlying track-categories, and it follows that we can lift $p$ to an object $\hat{p} \in \Pi_1\C[\hat{x},\hat{y}]$, for some lift $\hat{y}$ of $y$.
	Further,
	since $f$ is a local fibration, we can lift $\alpha$ to some element $\hat{\alpha} \in \C[\hat{x}, \hat{y}](\hat{p})$. Now $\hat{a} = (\hat{p}, \hat{\alpha})$ represents an an isomorphism in $\pi_0\dgbaradjoint(\C)$, by Lemma \ref{lem:all-0-cells-isomorphisms}, and it clearly lifts $a$, as desired. 
\end{proof}

\begin{theorem}\label{thm:dg-bar-construction-is-quillen}
	The adjunction
	\adjunctiondiagram{\dgbarconstruction}{\trackdgcataug}{\trackchaincat}{\dgbaradjoint}
	is a Quillen adjunction.
\end{theorem}
\begin{proof}
	We apply Theorem \ref{thm:naive-fibration-transfer}. 
	The hypotheses are satisfied as a result of Theorems \ref{thm:dgbaradjoint-preserves-naive-fibrations} and \ref{thm:track-dg-cat-naive-fibration}.
\end{proof}

\subsection{Underlying track-categories}
We observe that both functors in the adjunction $\dgbarconstruction \dashv \dgbaradjoint$ preserve the groupoid part of each hom-object:
\begin{definition}
	Let $\Pi_1: \ch^\times \to \gpd$ be the functor which forgets the chain complex, and similarly for $\Pi_1: \chaug^\otimes \to \gpd$. 
	These functors are strong monoidal left adjoints; denote by $\catfunctor{\Pi_1}$ both of the resulting functors 
	$$\trackdgcataug \to \gpdcat \quad \text{ and } \quad \trackchaincat \to \gpdcat$$
\end{definition}
\begin{theorem}
	The following diagrams commute:
	\begin{center}
	\begin{tikzcd}[sep=small]
		\trackdgcataug 
			\ar[rr, "\dgbarconstruction"] 
			\ar[rd, "\catfunctor{\Pi_1}" swap]
		&&
		\trackchaincat \ar[ld, "\catfunctor{\Pi_1}"]
		\\
		& 
		\gpdcat
	\end{tikzcd}
	\begin{tikzcd}[sep=small]
		\trackdgcataug 
			\ar[rd, "\catfunctor{\Pi_1}" swap]
		&&
		\trackchaincat 
			\ar[ld, "\catfunctor{\Pi_1}"]
			\ar[ll, "\dgbaradjoint " swap]
		\\
		& 
		\gpdcat 
	\end{tikzcd}
	\end{center}
\end{theorem}	
\begin{proof}
	The second diagram is easy to verify, since all three functors $\dgbaradjoint$, $\catfunctor{\Pi_1}$, and $\catfunctor{\Pi_1}$ are given by applying functors locally. 
	For the second diagram, it suffices to verify the corresponding diagram of right adjoints commutes, which is easy to do.
\end{proof}

\section{The square-zero adjunction}
\label{section:bar-construction}
By viewing chain complexes as elements of $\ch$, with groupoid part the terminal groupoid, we can view (nonnegatively homologically graded) $\dg$-categories as track $\dg$-categories.
We can thus apply $\dgbarconstruction$ to obtain a track chain-category. Since $\dgbarconstruction$ preserves the groupoid part of the hom-objects, the resulting track chain-category will be equivalently a category enriched in $\chbasic$, which by Theorem \ref{thm:cocartesian-enrichment-trivial} is equivalently an ordinary chain complex. 
The goal of this section is to prove that for an augmented $\dg$-category $\C$, this recovers the noncommutative $1$-forms/square-zero extension adjunction of \cite[Section 5]{Tabuada2009}.
\subsection{Square-zero extensions of $\dg$-categories} 
We recall briefly some of the main definitions of \cite{Tabuada2009}.
\begin{definition}
	Let $\A$ be a $\dg$-category. Then an \emph{$\A$-bimodule} $M$ is for every $x,y \in \ob \C$ a chain complex $M(x,y)$, natural in $x$ and $y$.
	The category of $\A$-bimodules is denoted $\A-\mathrm{bimod}$.
\end{definition}
More formally, an $\A$-bimodule is a functor $\A \otimes \A^{\mathrm{op}} \to \chbasic$.
The only example we will be concerned about here is the following:
\begin{example}
	Let $\A$ be the $\dg$-category with one object, and unique mapping complex $\mathbb{Z}[0]$. Then an $\A$ bimodule is the same thing as a chain complex; that is, there is an equivalence of categories
	$$\A-\mathrm{bimod} \simeq \chbasic$$
	More generally, for $S$ any set, there is a $\dg$-category $\mathbb{Z}[0]_S$ with object set $S$ where every hom-object is $\mathbb{Z}[0]$, with composition maps given by the algebra structure of $\mathbb{Z}[0]$. 
	Then a $\mathbb{Z}[0]_S$-bimodule is a functor from the contractible groupoid on the object set $S$ into $\chbasic$, and so we have a similar equivalence of categories 
	$$\mathbb{Z}[0]_S-\mathrm{bimod} \simeq \chbasic$$
\end{example}
Recall the following adjunction from \cite[Section 5]{Tabuada2009}:
\begin{definition}
	Let $S$ be a set. Denote by $\dg\text{-}\mathrm{cat}_S$ the category of $\dg$-categories with fixed object set $S$. 
	Let $\A \in \dg\text{-}\mathrm{cat}_S$, and denote by $\dg\text{-}\mathrm{cat}_S / \A$ the slice category.
	There is an adjunction 
	\begin{center}
	\begin{tikzcd}[sep = large]
		\dg\text{-}\mathrm{cat}_S / \A
			\ar[rr, shift left = 1, bend left = 5, "\Omega^1_{nc}(-)"]
		&&
		\A-\mathrm{bimod}
			\ar[ll, shift left = 1, bend left = 5, "A \ltimes -"]
	\end{tikzcd}
	\end{center}
	In which the right adjoint $\A \ltimes -$ takes a bimodule $M$ to the square-zero extension of $\A$ by $M$, i.e. $\A \ltimes M$ is the $\dg$-category with object set $S$ and $\A \ltimes M [x,y] := \A[x,y] \oplus M(x,y)$.
	The composition operation is therefore a map
	$$\left(\A[x,y] \oplus M(x,y)\right) \otimes \left( \A[y,z] \oplus M(y,z) \right) 
	\to 	\left(\A[x,z] \oplus M(x,z)\right)  
	$$
	which is equivalently a map 
	\begin{center}
	\begin{tikzcd}
		\left(\A[x,y] \otimes \A[y,z]\right) \oplus \left(\A[x,y] \otimes M(y,z) \right) \oplus \left(M(x,y) \otimes \A[y,z]\right) \oplus \left(M(x,y) \otimes M(y,z) \right)
			\ar[d]
		\\
		\A[x,z] \oplus M(x,z)
	\end{tikzcd}
	\end{center}
	Which is defined to be composition in $\A$ in the first component, the bimodule action of $\A$ on $M$ in the second and third component, and $0$ in the last component.
\end{definition}
In the special case that $\A = \mathbb{Z}[0]_S$, then via the equivalence of categories $\A-\mathrm{bimod} \simeq \chbasic$, this gives us an adjunction
\begin{center}
	\begin{tikzcd}[sep = large]
		\dg\text{-}\mathrm{cat}_S / \mathbb{Z}[0]_S
			\ar[rr, shift left = 1, bend left = 5, "\Omega^1_{nc}(-)"]
		&&
		\chbasic
			\ar[ll, shift left = 1, bend left = 5, "A \ltimes -"]
	\end{tikzcd}
\end{center}
\begin{definition}
	Denote by $\dg_*\text{-}\mathrm{cat}$ the category of categories enriched in augmented chain complexes (with respect to $\otimes$). 
	Similarly, for $S$ a set, denote by $\dg_*\text{-}\mathrm{cat}_S$ the subcategory of $\dg_*\text{-}\mathrm{cat}$ with fixed object set $S$ (and morphisms only those functors which are identity-on-objects).
\end{definition}
It is straightforward to check that $\dg_*\text{-}\mathrm{cat}$ is equivalent to the slice category $\dg\text{-}\mathrm{cat}/\mathbb{Z}[0]_{\{0\}}$, and that similarly $\dg_*\text{-}\mathrm{cat}$ is equivalent to $\dg\text{-}\mathrm{cat}_S/\mathbb{Z}[0]_S$.

\subsection{$\dg$-categories and track $\dg$-categories}
\newcommand{\bchbasic}{\mathcal{B}\chbasic}

\begin{definition}
	Let $T \in \gpdcat$. 
	Define a category $[T, \bchbasic]^\otimes$, in which:
	\begin{itemize}
		\item The objects are augmented track $\dg$-categories with track category part exactly $T$, and 
		\item The morphisms are morphisms of track $\dg$-categories which are identity-on-objects and identity on groupoid parts.
	\end{itemize}
\end{definition}

\begin{definition}
	Let $T \in \gpdcat$. 
	Define a category $[T, \bchbasic]^\times$, in which:
	\begin{itemize}
		\item The objects are track chain-categories with track category part exactly $T$, and 
		\item The morphisms are morphisms of track $\dg$-categories which are identity-on-objects and identity on groupoid parts.
	\end{itemize}
\end{definition}
The notation $[T, \bchbasic]$ comes from thinking of this as a category of lax functors between two $2$-categories: the $(2,1) $-category $T$, and the classifying $2$-category $\bchbasic$ of the monoidal category $\chbasic$. 
We imagine that one could recover the model categories on $\trackdgcataug$ and $\trackchaincat$ by viewing them as Grothendieck constructions $\int_{\gpdcat}[-, \bchbasic]$, but we will not explore that here.
\begin{definition}\label{dfn:K(S)}
	Let $S \in \set$. 
	Define $K(S) \in \gpdcat$ to be the track-category with object set $S$, and every hom-object the terminal groupoid. Define $p: K(S)^\cof \to K(S)$  to be a cofibrant replacement of $K(S)$ (in $\gpdcat$). 
\end{definition}
Clearly, $[K(S), \bchbasic]^\otimes$ is simply the category of ordinary $\dg$-categories with object set $S$.
By Theorem \ref{thm:cocartesian-enrichment-trivial}, $[K(S), \bchbasic]^\times$ is equivalent to the category $\chbasic$.
\par 
Since both functors of the adjoint pair $\dgbarconstruction \dashv \dgbaradjoint$ are identity on underlying track-categories, they restrict to give an adjunction 
$$
\left(\dgbarconstruction\right)_T : [T, \bchbasic]^\otimes \rightleftarrows [T, \bchbasic]^\times : \left(\dgbaradjoint\right)_T
$$
for a fixed track-category $T$. In particular, for $T = K(S)$, we get an adjunction which fits into a diagram
\begin{equation}\label{diagram:functors-dg-bar}
\begin{tikzcd}[sep = large]
	\trackdgcataug
		\ar[r, "\dgbarconstruction"]
	&
	\trackchaincat
	\\
	{[K(S), \bchbasic]}^\otimes 
		\ar[u, hook]
		\ar[r, "\left(\dgbarconstruction\right)_{K(S)}"]
	&
	{[K(S), \bchbasic]}^\times
		\ar[u, hook]
	\\
	\dg\text{-}\cat_S
		\ar[r, "\Omega^1_{nc}(-)"]
		\ar[u, "\sim" labl]
	&
	\chbasic
		\ar[u, "\sim" labl]
\end{tikzcd}
\end{equation} 
\begin{theorem}\label{thm:noncommutative-forms-is-dg-bar}
	Diagram \eqref{diagram:functors-dg-bar} commutes.
\end{theorem}
\begin{proof}
	The top square clearly commutes, so it suffices to show the bottom square commutes. It is equivalent to show that the diagram of right adjoints
	\begin{equation}
	\begin{tikzcd}[sep = large]
	{[K(S), \bchbasic]}^\otimes 
	&
	{[K(S), \bchbasic]}^\times
		\ar[l, "\dgbaradjoint_T" swap]
	\\
	\dg_S\text{-}\cat
		\ar[u, "\sim" labl]
	&
	\chbasic
		\ar[l, "{ \mathbb{Z}[0] \ltimes -}"]
		\ar[u, "\sim" labl]
	\end{tikzcd}
	\end{equation}
	commutes. 
	Since $\dgbaradjoint$ is given by applying $\mathbb{Z}[0] \oplus -$ locally, this is fairly straightforward.
\end{proof}

\begin{remark}
	The nature of our construction is such that we cannot expect to reproduce in an enriched fashion the adjunction $\Omega^1_{nc}(-) \dashv - \ltimes \A$ for a general $\dg$-category $\A$. 
	Here we have argued that we reproduce it in the case where $\A$ is the initial $\dg$-algebra; however, our arguments would work to reproduce it in the more general case of any $\dg$-algebra $\A$ .
	To do so, we would enrich in the tensor monoidal category of chain complexes over $\A$ on the one hand, and the cartesian monoidal category of $\A$-modules on the other. 
\end{remark}

\subsection{Compatibility with the derived functor}
Theorem \ref{thm:noncommutative-forms-is-dg-bar} says that for a $\dg$-category $\C$ over the $\dg$-category $\mathbb{Z}[0]_{\ob \C}$, viewing $\C$ as an augmented track $\dg$-category and applying $\dgbarconstruction$ is equivalent to the functor $\Omega^1_{nc}$ of \cite{Tabuada2009}. 
However, it is not immediate that the \textit{derived} functors agree: 
the immediate obstruction is that a cofibrant augmented $\dg$-category will almost never be cofibrant as an element of $\trackdgcataug$, since the underlying track-category of a cofibrant track $\dg$-category must be cofibrant, and in particular the underlying category must be free. 
We show in this section that for $\C$ a cofibrant augmented $\dg$-category, applying $\dgbarconstruction$ directly to $\C$ gives (up to weak equivalence) the same result as applying $\dgbarconstruction$ to a cofibrant replacement of $\C$.
\begin{definition}\label{dfn:pushforward-functor}
	Let $S$ be a set, and $K(S)$ and $K(S)^\cof$ as in definition \ref{dfn:K(S)}. Let $\boxtimes = \otimes$ or $\boxtimes = \times$.
	Define a functor $p_*: [K(S)^\cof, \bchbasic]^\boxtimes \to [K(S), \bchbasic]^\boxtimes$ to be the identity-on-objects functor with $p_*(\C)[s,t]$ defined by
	$$\bullet \mapsto \colim_{\Pi_1\C[s,t]}\C[s,t]$$
	where $\bullet$ is the unique object of $p_*(\C)[s,t]$.
\end{definition}

\begin{definition}\label{dfn:pullback-functor}
	Let $S$ be a set, and $K(S)$ and $K(S)^\cof$ as in definition \ref{dfn:K(S)}. Let $\boxtimes = \otimes$ or $\boxtimes = \times$.
	Define a functor $p^*: [K(S), \bchbasic]^\boxtimes \to [K(S)^\cof, \bchbasic]^\boxtimes$ to be the identity-on-objects functor for which the functor $(p^* \C)[s,t]$ is defined by the composition
	$$p^*(\C)[s,t] \to \C[s,t] \to \chbasic$$
	In other words, the value on any object of $p^*(\C)[s,t]$ is the chain complex $\C[s,t](\bullet)$, and the value on any morphism is the identity.  
\end{definition}
The following theorem is immediate.
\begin{theorem}	
	Let $S$ be a set, and $K(S)$ and $K(S)^\cof$ as in definition \ref{dfn:K(S)}. Let $\boxtimes = \otimes$ or $\boxtimes = \times$. The functors $p_* :  [K(S)^\cof, \bchbasic]^\boxtimes \rightleftarrows [K(S), \bchbasic]^\boxtimes : p^*$ are adjoint equivalences.
\end{theorem}

\begin{lemma}
	The following diagram of left adjoint functors commutes up to natural isomorphism:
	\begin{center}
	\begin{tikzcd}[sep = large]
		{[K(S)^\cof, \bchbasic]}^\otimes
			\ar[rr, "\left(\dgbarconstruction\right)_{K(S)^\cof}"]\
			\ar[d, "p_*"]
		&&
		{[K(S)^\cof, \bchbasic]}^\times
			\ar[d, "p_*"]
		\\
		{[K(S), \bchbasic]}^\otimes
			\ar[rr, "\left(\dgbarconstruction\right)_{K(S)}"]
		&&
		{[K(S), \bchbasic]}^\times
	\end{tikzcd}
	\end{center}
\end{lemma}
\begin{proof}
	It is easy to check that the diagram of right adjoints commutes.
\end{proof}
We therefore obtain:
\begin{theorem}\label{thm:dgbar-is-derived}
	Let $S$ be a set, and $\C \in [K(S), \bchbasic]$.
	The map $p^*\C \to \C$ (definition \ref{dfn:pullback-functor}) is sent by $\dgbarconstruction$ to a map isomorphic to the pullback $p^*\dgbarconstruction \C \to \dgbarconstruction \C$.
\end{theorem}
And from this, the desired result follows:
\begin{corollary}
	Let $\C$ be a cofibrant augmented $\dg$-category. 
	Viewing $\C$ as an augmented track $\dg$-category, $\dgbarconstruction(\C)$ computes the value of the derived functor of $\dgbarconstruction$ on $\C$.
\end{corollary}
\begin{proof}
	Since the generating cofibrations of $\trackdgcataug$ agree with those of $\dg\text{-}\cat$, it follows that if $\C$ is a cofibrant $\dg$-category, then $p^*\C$ is cofibrant as a track $\dg$-category. The result now follows from Theorem \ref{thm:dgbar-is-derived}.
\end{proof}
In particular, the derived functor of $\dgbarconstruction$ coincides with the derived functor of $\Omega^1_{nc}$ on the subcategory of augmented $\dg$-categories.
\section{Endomorphism monoids of intervals and Theorem \ref{alphathm:complicated}}\label{section:thmB}
In this section, we prove Theorem \ref{alphathm:complicated} (Theorem \ref{thm:main-theorem-chapter-two}), which will require us to do some analysis of pushouts of enriched categories.
Throughout this section, $\V$ and $\W$ are closed monoidal model categories, and we assume the unit $\iv$ of $\V$ is cofibrant.
We assume we have a Quillen adjunction
\adjunctiondiagram{L}{\V}{\W}{R}
with $R$ lax monoidal, which induces a (not necessarily Quillen) adjunction
\adjunctiondiagram{\lcat}{\vcat}{\wcat}{\rcat}
For $\C \in \vcat$ and $x,y \in \ob \C$, we use the notation $\C[x,y]$ for the corresponding hom-object $\Hom_\C(x,y)$. We use the notation $\C_x$ for the hom-object $\C[x,x]$.
We write our composition operations in left-to-right order; e.g. composition is an operation $\C[x,y] \otimes \C[y,z] \to \C[x,y]$.
\begin{definition}\label{dfn:suspension-category}
	For $X \in \V$, we define $\catsus X \in \vcat$ to be the $\V$-category with object set $\{0,1\}$, and with hom-objects given as follows:
		$$\catsus X[0,0] = \catsus X[1,1] = \iv \quad \quad \catsus X[1,0] = \emptyset \quad \quad \catsus X[0,1] = X$$
\end{definition}
\begin{definition}\label{dfn:multi-suspension-category}
	For $X_1,...,X_n$ a list of objects in $\V$, we define $\catsus(X_1,\dots, X_n)$ to be the $\V$-category given as the pushout
	\begin{center}
	\begin{tikzcd}[sep = small]
	& \bullet \ar[ld, "1", swap] \ar[rd, "0"] & & \bullet \ar[ld, "1", swap] \ar[rd, "0"] & & \cdots & & \bullet \ar[ld, "1", swap] \ar[rd, "0"] 
	\\
	\catsus X_1 & & \catsus X_2 & & \catsus X_3 & & \catsus X_{n-1} & & \catsus X_n
	\end{tikzcd}
	\end{center}
	where $\bullet$ is the initial $\V$-category with one object. 
	We label the objects of this pushout $\{0,\dots,n\}$. For example, $\catsus(X_1,\dots,X_n)[i - 1, i] \cong X_i$.
\end{definition}

\begin{lemma}\label{lem:catsus-preserved-by-adjoint}
	For an object $X$ in $\V$, there is a natural isomorphism
	$$\lcat \left( \catsus X \right) \cong \catsus (L X)  $$
	natural in $X$. Similarly, for a list of objects $X_1,...,X_n \in \V$, there is a natural isomorphism
	$$\lcat \left( \catsus(X_1,...,X_n) \right) \cong \catsus(L X_1, \dots, LX_n)$$
\end{lemma}
\begin{proof}
	The second statement follows from the first, as $\lcat$ preserves pushouts and $\bullet$. It remains to show the first statement, which follows by observing that $\lcat \left( \catsus X \right)$ and  $\catsus (L X)  $ satisfy the same universal property:
	Using the adjunction, it is easy to see that these satisfy the same universal property: a map from either into a $\W$-category $\D$ is equivalent to the choice of two objects $a,b \in \ob \D$ and a map $LX \to \D[a,b]$.
\end{proof}
For objects $X,Y$ is straightforward to check that $\sus(X,Y)[0,2] \cong X \otimes Y$, inducing a map $\sus(X \otimes Y) \to \sus(X,Y)$ which is a local isomorphism. 
Understanding the behavior of $\lcat$ on this map is a primary step in understanding the behavior of $\lcat$ on pushouts more generally. 
The following lemma will therefore be key in the proof of our main theorem: 
\begin{lemma}\label{lem:join-of-intervals-is-colax-structure-map}
	 Let $X,Y \in \V$. 
	 Let $f: \sus(X \otimes Y) \to \sus(X,Y)$ be the map induced by the identification $\sus(X,Y)[0,2] \cong X \otimes Y$. 
	 Then the map $\lcat(f)[0,1]: \lcat(\sus (X \otimes Y))[0,1] \to \lcat(\sus(X,Y))[0,2]$ is isomorphic to the colax structure map 
	 $$L(X \otimes_V Y) \to L(X) \otimes_W L(Y)$$
	 More generally, for a list of objects $X_1,\cdots,X_n$ in $\V$, the evident map
	 $$\lcat(\sus (X_1 \ov X_2 \ov \cdots \ov X_n))[0,n] \to \lcat(\sus(X_1,X_2, \dots, X_n))[0,n]$$
	 is isomorphic to the colax structure map 
	 $$L(X \ov X_2 \ov \cdots \ov X_n) \to L(X_1) \ow L(X_2) \ow \cdots \ow L(X_n)$$
	 ``In the case $n=0$,'' we have that the map
	 $$\lcat(\sus \iv) \to \{0\}_\W $$
	 (where $\{0\}_\W$ is the initial $\W$-category on one object) is given by the colax structure map $L(\iv) \to \iw$
\end{lemma}
\begin{proof}
	The isomorphisms $\lcat(\sus (X \otimes Y))[0,1] \cong L(X \otimes_V Y)$ and $\lcat(\sus(X,Y))[0,2] \cong L(X) \otimes_W L(Y)$ are from Lemma \ref{lem:catsus-preserved-by-adjoint}.
	The identification with the colax structure will come from considering the adjoint of $\lcat(f)[0,1]$; which is given by the composition of $f$ with the unit of the adjunction at $\sus(X,Y)$:
	From the definition of $\rcat$, it follows that the adjoint to $\lcat(f)[0,1]$ is the composition 
	$$X \otimes Y \to  RLX \otimes RL Y \to R(LX \otimes LY)$$
	with the first map induced by the unit of the adjunction and the second by the lax structure of $R$. The adjoint to this is by definition the colax structure map of $L$, as desired.
	\par 
	The statement for arbitrary $n$ follows in the same way. 
	The ``$n = 0$'' case is more obvious; this is simply stating that the map $\lcat(\sus \iv) \to \{0\}_W $ is adjoint to the map $\sus \iv \to \rcat \left (\{ 0\}_\W \right)$. 
\end{proof}
\subsection{Generating intervals}
Let $C_1$ and $C_2$ be cylinder objects for the unit $\iv$. From these cylinders, we can construct an interval (in the sense if \cite{Muro2012}, NOT in the sense of \cite{BergerMoerdijk2012}) $\mathbb{I}(C_1,C_2)$ given as follows:
we define $\H$ by $\sus\iv \coprod (\sus\iv)^{\text{op}}$, where the coproduct is taken in the category of $\V$-categories with fixed object set $\{0,1\}$.
Denote by $a$ and $b$ the canonical maps $\iv \to \H[0,1] $ and $\iv \to \H[1,0]$, respectively.
We then define $\K$ and $\I(C_1,C_2)$ in $\vcat$ by the following pushouts in $\vcat$:
\begin{equation}\label{diagram:pushout-description-intervals}
\begin{tikzcd}
	\T(\iv \coprod \iv) \ar[d] \ar[r, "\id_1 \coprod b\otimes a"]
		& \H \ar[d]
	\\
	\T(C_2) \ar[r]
		& \K
\end{tikzcd}
\quad \quad \quad 
\begin{tikzcd}
	\T(\iv \coprod \iv) \ar[d] \ar[r, "\id_0 \coprod a\otimes b "]
		& \K \ar[d]
	\\
	\T(C_1) \ar[r]
		& \I(C_1,C_2)
\end{tikzcd}
\end{equation}
Roughly, $\I(C_1,C_2)$ is free on generators $a: 0 \to 1$ and $b: 1 \to 0$, and on homotopies $C_2: \id_1 \Rightarrow b \otimes a$ and $C_1: \id_0 \Rightarrow a \otimes b$.
For example, if $\V$ is equipped with the trivial model structure, then the cylinders $C_1$ and $C_2$ are isomorphic to $\iv$, and it is easy to check that $\I(C_1,C_2) \cong \I_\V$ (as in Definition \ref{dfn:interval-berger-moerdijk}). 
For the rest of this section, we will analyze the structure of the intervals $\I(C_1,C_2)$, in particular the endomorphism monoid $\I(C_1,C_2)_0$. 
We will heavily draw on the analysis of pushouts in $\vcat_{\{0,1\}}$ found in \cite[Section 3]{BergerMoerdijk2012}.
\par
Recall that for $\I$ a $\V$-category with object set $\{0,1\}$, we denote by $\I_0$ and $\I_1$ the endomorphism monoids of the objects $0$ and $1$, respectively.
Our main goal is to describe $\I(C_1,C_2)_0$ as a (homotopy) pushout of monoids in $\V$.
For the rest of this subsection, we fix two cylinders $\iota_1 :\iv \coprod \iv \hookrightarrow C_1$ and $\iota_2: \iv \coprod \iv \hookrightarrow C_2$, and let $\H$ and $\K$ be as in diagram \eqref{diagram:pushout-description-intervals}.
First, we will describe $\K_0$ as a pushout of monoids.
\par 
Let $n > 0$. For $P \subseteq \n$, define an object $X^{(n)}_P$ of $\V$ by
$$X^{(n)}_P = \iv \otimes A_1 \otimes A_2 \otimes \cdots \otimes A_n \otimes \iv$$
where 
$$
A_i = 
\begin{cases}
	C_2 & \text{ if } i \in P \\
	\iv \coprod \iv & \text{ if } i \not \in P
\end{cases}
$$
The natural maps $X^{(n)}_P \to X^{(n)}_Q$ for $P \subseteq Q$ (induced by the map $\iota_2: \iv \coprod\iv \to C_2$) give rise to a diagram with shape an $n$-dimensional cube and objects the $X^{(n)}_P$.
Define $X^{(n)}_{-}$ to be the colimit of the full subdiagram consisting of all of the objects except $X^{(n)}_{\n}$. 
We define $X^{(n)}$ to be $X^{(n)}_{\n}$.
Now, inductively define monoids $M^{(n)}(C_1,C_2)$ as follows: 
$M^{(0)}(C_1,C_2) = \H_0 \cong \tiv$. For $n > 0$, $M^{(n)}(C_1,C_2)$ is the following pushout of monoids: 
\begin{equation}\label{diagram:homotopy-pushout-step}
\begin{tikzcd}
\T(X^{(n)}_{-}) \ar[d] \ar[r] 
	& M^{(n-1)}(C_1,C_2) \ar[d]
\\
\T\left(X^{(n)} \right)\ar[r]
	& M^{(n)}(C_1,C_2)
\end{tikzcd}
\end{equation}
We have not yet defined the map $\T(X^{(n)}_{-}) \to M^{n}(C_1,C_2)$, and we do so now: for $n= 1$, we need to define a map $\iv\otimes (\iv \coprod \iv) \otimes \iv \to \H_0$, which we define as the composite 
\begin{center}
\begin{tikzcd}[sep = huge]
\iv \otimes (\iv \coprod \iv) \otimes \iv 
	\ar[r, "a \otimes (\id \coprod ab) \otimes b"]
&
\H[0,1] \otimes \H_1 \otimes \H[1,0] 
	\ar[r, "\circ "]
&
\H_0
\end{tikzcd}
\end{center}
where $\circ$ is the composition map in $\H$. 
Now for $n > 1$, we inductively define the map $X^{(n)}_{-} \to M^{(n-1)}(C_1,C_2)$ as follows: we must define compatible maps $X^{(n)}_P \to M^{(n)}(C_1,C_2)$, for $P \subseteq \n$.
Suppose that $ i \not \in P$. Then we can write
$$X^{(n)}_P = \iv \otimes A_1 \otimes \cdots \otimes (\iv \coprod \iv) \otimes \cdots \otimes A_n \otimes \iv 
$$
Which is isomorphic to the coproduct
$$ 
\iv \otimes A_1 \otimes \cdots \otimes A_{i-1 } \otimes \iv \otimes A_{i + 1} \otimes \cdots \otimes A_n \otimes \iv 
\coprod 
\iv \otimes A_1 \otimes \cdots \otimes A_{i-1 } \otimes \iv \otimes A_{i + 1} \otimes \cdots \otimes A_n \otimes \iv 
$$
We map each component in inductively: the first component we  rewrite as 
$$\iv \otimes A_1 \otimes \cdots \otimes A_{i-1 } \otimes A_{i + 1} \otimes \cdots \otimes A_n \otimes \iv 
$$
and think of as being $X^{(n-1)}_{d_i(P)}$. The second component we write as 
$$\left( \iv \otimes A_1 \otimes \cdots \otimes A_{i-1} \otimes \iv \right) \otimes \left( \iv \otimes A_{i+1} \otimes \cdots \otimes A_n \otimes \iv \right)
$$
And we map each half separately into $M^{(n-1)}(C_1,C_2)$ via the natural maps $X^{(i)} \to M^{i}(C_1,C_2) \to M^{(n-1)}(C_1,C_2)$ and $X^{(n-i)} \to M^{(n-i)}(C_1,C_2) \to M^{(n-1)}(C_1,C_2)$, obtaining a map on the tensor product using the multiplication in $M^{(n-1)}$. 
We omit checking that that these maps $X^{(n)}_P$ compatibly assemble to give a map out of the colimit $X^{(n)}_{-}$, which follows inductively (i.e. inductively applying the fact that that diagram \eqref{diagram:homotopy-pushout-step} commutes).
\par 
Before moving on, we note that this description of $M^{(n)}(C_1,C_2)$ is in fact a homotopy pushout description:
\begin{lemma}\label{lem:homotopy-pushout-monoid}
	The diagram \eqref{diagram:homotopy-pushout-step} is a homotopy pushout.
\end{lemma}
\begin{proof}
	It suffices to verify that one of the maps is a cofibration, and that the target of the other is cofibrant. Indeed, all $3$ objects are cofibrant (by cofibrancy of $\iv$), and the vertical map is a cofibration, by the pushout product axiom.
\end{proof}
With this defined, we can finally define $M(C_1,C_2)$ to be the directed colimit of the $M^{(n)}(C_1,C_2)$. Our goal is now to define a map $\psi: M(C_1,C_2) \to \K_0$, and show it is an isomorphism.
We can inductively define maps $M^{(n)}(C_1,C_2) \to \K_0$ by defining maps in $\V$
$$\iv \otimes C_2 \otimes \cdots \otimes C_2 \otimes \iv  \to \K_0$$
compatible with the pushout diagram \eqref{diagram:homotopy-pushout-step}.
We define these via $\circ (a \otimes \iota_2 \otimes \cdots \otimes \iota_2 \otimes b)$ (i.e. using the composition operation in the category $\K$); verifying the compatibility is again straightforward.
Assembling these together yields a map of monoids $\psi: M(C_1,C_2) \to \K_0$. 
\par 
We will now define a morphism $\phi: \K_0 \to M(C_1,C_2)$ in $\V$ (that is, a morphism \textit{a priori} of objects, not necessarily of monoids) and show that $\phi$ and $\psi$ are inverse in $\V$. 
Since the forgetful functor from monoids in $\V$ to $\V$ reflects isomorphisms, it will then follow that $\psi$ is an isomorphism of monoids.
\par
First, let us introduce a little bit of notation: for $n > 0$, define $\chi_n : \iv \otimes C_2^{\otimes n} \otimes \iv \to M$ to be the adjoint to the map $\T(X^{(n)}) \to M^{(n)}(C_1,C_2) \to M(C_1,C_2)$. 
Define a map $\chi_{ab}: \iv \to M$ by the composition 
\begin{center}
\begin{tikzcd}[sep = small]
\iv \ar[r, "\sim"] 
	& \iv \otimes \iv \ar[r, "a \otimes b"]
	& \H[0,1] \otimes \H[1,0] \ar[r, "\circ "]
	& \H_0 \ar[r, hook]
	& M
\end{tikzcd}
\end{center} 
Denote by $\miv(\iv)$ the $\tiv$-module $\tiv \otimes \iv $, which has underlying object isomorphic to the underlying object of $\tiv$, i.e.
$$|\tiv| \cong |\miv| \cong \iv \coprod \iv^{\otimes^2} \coprod \cdots \cong \iv \coprod \iv \coprod \cdots$$ 

Recall that by \cite{BergerMoerdijk2012}, $\K_0$ is the directed colimit of objects $\K_0^{(i)}$, where the $\K_0^{(i)}$ are given inductively as pushouts in $\V$
\begin{center}
\begin{tikzcd}
Y^{(n)}_- \ar[d] \ar[r] 
	& \K_0^{(n-1)} \ar[d]
\\
Y^{(n)} \ar[r]
	& \K_0^{(n)}
\end{tikzcd}
\end{center}
Where $Y^{(n)} = \tiv \otimes C_2 \otimes \miv \otimes C_2 \otimes \cdots \otimes C_2 \otimes \tiv$, and $Y^{(n)}_-$ is the colimit over objects like $Y^{(n)}$ but with at least one $C_2$ replaced by $\iv \coprod \iv$.
Thus to define $\phi$, it suffices to define compatible maps $Y^{(n)} \to M(C_1,C_2)$. 
Note that as objects of $\V$, the $Y^{(n)}$ are given as disjoint unions of copies of $C_2^{\otimes n}$, since $\tiv$ and $\miv$ are each disjoint unions of copies of $\iv$.
In particular, this disjoint union is naturally indexed on $\mathbb{N}^{+} \times \mathbb{N} \times \cdots \times \mathbb{N} \times \mathbb{N}^+$. 
Thus, we need to define maps $\phi_{(j,k_1,...,k_n,\ell)}: C_2^{\otimes n} \to M(C_1,C_2)$,
for $(j,k_1,..., k_n, \ell) \in \N^+ \times \N \times \cdots \times \N \times \N^+$.
Define indices $\alpha_i$ by 
$$\alpha_i = \text{the }i\text{th nonzero entry of }(k_1,...,k_n)$$
Let $K$ be the number of nonzero entries of $(k_1,...,k_n)$. 
Then we define $\phi_{(j,k_1,...,k_n, \ell)}$ as the composition  
\begin{center}
\begin{tikzcd}
\iv 
	\otimes C_2^{\alpha_1} 
	\otimes \iv 
	\otimes C_2^{\alpha_2 - \alpha_1} 
	\otimes \iv 
	\otimes \cdots 
	\otimes C_2^{\alpha_{K} - \alpha_{K - 1}}
	\otimes \iv
\ar[d, "
	\chi_{ab}^{j - 1} 
	\otimes \chi_{\alpha_1} 
	\otimes \chi_{ab}^{k_{\alpha_1} - 1} 
	\otimes \chi_{\alpha_2 - \alpha_1}
	\otimes \chi_{ab}^{k_{\alpha_2} - 1} 
	\otimes \cdots 
	\otimes \chi_{ab}^{k_{\alpha_K} - 1}
	\otimes \chi_{ab}^{\ell - 1}
"]
\\
	M(C_1,C_2)
	\otimes M(C_1,C_2)
	\otimes M (C_1,C_2)
	\otimes M (C_1,C_2)
	\otimes M (C_1,C_2)
	\otimes \cdots 
	\otimes M(C_1,C_2)
	\otimes M(C_1,C_2)
\ar[d, "\mu "] \\
M(C_1,C_2)
\end{tikzcd}
\end{center}
where $\mu$ is the multiplication map of the monoid $M(C_1,C_2)$. 
It is easily (if somewhat tediously) verified that the various diagrams commute.
Thus, this defines our map $\phi$.
\begin{theorem}\label{thm:pushout-description-K0}
	The maps $\psi: M(C_1,C_2) \to \K_0$ and $\phi: \K_0 \to M(C_1,C_2)$ are inverse.
\end{theorem}
\begin{proof}
	Showing $\phi \circ \psi: M(C_1,C_2) \to M(C_1,C_2)$ is the identity boils down to analyzing the behaviour on the generating components $X^{(n)}$.
	Showing $\psi \circ \phi: \K_0 \to \K_0$ is the identity is equally straightforward, and involves checking that every $C_2^{\otimes n}$ component is mapped to itself via the identity.
\end{proof}
Finally, we can complete our initial goal of describing $\I(C_1,C_2)_0$ as a pushout: 
\begin{lemma}\label{lem:pushout-description-interval}
	$\I(C_1,C_2)_0$ is the pushout of monoids 
	\begin{center}
	\begin{tikzcd}
		\T(\iv \coprod \iv) \ar[d] \ar[r, "\id \coprod \chi_{ab} "]
			& M(C_1,C_2) \ar[d]
		\\
		\T(C_1) \ar[r]
			& \I(C_1,C_2)_0
	\end{tikzcd}
	\end{center}
	And this pushout is a homotopy pushout.
\end{lemma}
\begin{proof}
	That this diagram is a pushout follows from the description of $\I(C_1,C_2)$ as a pushout, the description of monoids of pushouts as given in \cite{BergerMoerdijk2012}, and the fact that $\psi: M(C_1,C_2) \to \K_0$ is an isomorphism.
	That it is furthermore a homotopy pushout is a consequence of the left vertical arrow being a cofibration, along with cofibrancy of all $3$ objects.
\end{proof}
\subsection{Images of generating intervals under $\lcat$}
For $C_1$ and $C_2$ cylinders for the unit $\iv$, the previous subsection gives us a formula for $\I(C_1,C_2)$ as a pushout of monoids, and hence allows us to compute $\lcat(\I(C_1,C_2)_0)$. 
We will need to know whether the map $\lcat(\I(C_1,C_2)_0) \to \left(\lcat\I(C_1,C_2)\right)_0$ is a weak equivalence, and hence we would like an understanding of $\left(\lcat\I(C_1,C_2)\right)_0$. Rather than repeat the analysis of the previous subsection, we will describe how to alter it to obtain the desired description.
\par 
Since $\lcat$ preserves pushouts, the object $\lcat\I(C_1,C_2)$ is given as a pushout given by applying $\lcat$ to diagram \eqref{diagram:pushout-description-intervals}:
\begin{equation}\label{diagram:pushout-description-L-intervals}
\begin{tikzcd}[sep = large]
	\T(L(\iv) \coprod L(\iv) ) \ar[d] \ar[r, "\tilde{\id}_1 \coprod \tilde{(b\otimes a)}"]
		& \lcat\H \ar[d]
	\\
	\T(L(C_2)) \ar[r]
		& \lcat\K
\end{tikzcd}
\quad \quad \quad 
\begin{tikzcd}
	\T(L(\iv) \coprod L(\iv)) \ar[d] \ar[r, "\tilde{\id}_0 \coprod \widetilde{a\otimes b}"]
		& \K \ar[d]
	\\
	\T(L(C_1)) \ar[r]
		& \lcat\I(C_1,C_2)
\end{tikzcd}
\end{equation}
The top horizontal maps in each diagram are somewhat non-obvious: the maps $\tilde{\id}_1$ and $\tilde{\id}_0$ are
induced by the colax structure maps $L(\iv) \to \iw$, e.g. $\tilde{\id}_1$ is adjoint to the composition 
\begin{center}
\begin{tikzcd}
	L(\iv) 
		\ar[r]
	& 
	\iw 
		\ar[r, "\id_1"]
	& 
	\H_1
\end{tikzcd}
\end{center}
where $\id_0$ is the identity structure map for the monoid $\H_0$. 
\par 
The maps $\tilde{a\otimes b}$ and $\tilde{b \otimes a}$ are induced using the isomorphism $\iv \to (\iv \otimes \iv)$ and Lemma \ref{lem:join-of-intervals-is-colax-structure-map}; e.g. $\tilde{b \otimes a}$ is the composition 
\begin{center}
\begin{tikzcd}
	L(\iv) 
		\ar[r, "\sim"]
	& 
	L(\iv \otimes \iv)
		\ar[r]
	& 
	L(\iv) \otimes L(\iv)
		\ar[r, "b \otimes a"]
	&
	\lcat\H[1,0] \otimes \lcat\H[0,1]
		\ar[r, "\circ"]
	& \H_1
\end{tikzcd}
\end{center}
where the maps $a: L(\ov) \to \H[0,1]$ and $b: L(\iv) \to \lcat\H[1,0]$ are given by applying $\lcat$ to the defining maps $\sus \iv \to \H$.
\par 
Note that since $\iv$ is cofibrant and $L$ is left Quillen, $L(\iv)$, $L(C_1)$, and $L(C_2)$ are all cofibrant. 
The same constructions as in the previous section give us a monoid in $\W$, $M(L(C_1),L(C_2))$, which is given as  a directed colimit of monoids $M^{(n)}(L(C_1),L(C_2))$ which fit into pushout diagrams
\begin{equation}
\begin{tikzcd}
\T(\tilde{X}^{(n)}_{-}) \ar[d] \ar[r] 
	& M^{(n-1)}(L(C_1),L(C_2)) \ar[d]
\\
\T\left(\tilde{X}^{(n)} \right)\ar[r]
	& M^{(n)}(L(C_1),L(C_2))
\end{tikzcd}
\end{equation}
where $\tilde{X}^{(n)}$ is the object 
$$L(\iv) \ow  L(C_2) \ow \cdots \ow L(C_2) \ow L(\iv)$$
and $\tilde{X}^{(n)}_{-}$ is the colimit of an $n$-dimensional cube of similar objects with $L(C_2)$ entries replaced by $L(\iv)$. 
Applying the same analyses as in the previous section, we deduce that the above diagram is also a homotopy pushout, and obtain:
\begin{lemma}\label{lem:pushout-description-L-interval}
	$\left(\lcat \I(C_1,C_2)\right)_0$ is the pushout of monoids 
	\begin{center}
	\begin{tikzcd}
		\T(L(\iv) \coprod L(\iv)) \ar[d] \ar[r]
			& M(L(C_1), L(C_2)) \ar[d]
		\\
		\T(L(C_1)) \ar[r]
			& \left(\lcat \I(C_1,C_2)\right)_0
	\end{tikzcd}
	\end{center}
	And this pushout is a homotopy pushout.
\end{lemma}
\subsection{Comparison of endomorphism monoids}
Given an interval $\I(C_1,C_2)$, we have the inclusion $\I(C_1,C_2)_0 \hookrightarrow \I(C_1,C_2)$ of the full subcategory on the object $0$. 
Applying $\lcat$ to this, we obtain (by definition of morphisms in $\wcat$)
$$\lcat\left(\I(C_1,C_2)_0\right) \to \left(\lcat \I(C_1,C_2) \right)_0 \hookrightarrow \lcat \I(C_1,C_2)$$
Our goal in this subsection is to gain a better understanding of this morphism.
\begin{lemma}\label{lem:last-pushout-comparison-step}
	The morphism 
	$$\lcat\left(\I(C_1,C_2)_0\right) \to \left(\lcat \I(C_1,C_2) \right)_0$$
	is induced as a morphism of pushouts via a transformation of diagrams 
	\begin{equation}\label{eqn:diagram-transformation-last-step}
	\begin{tikzcd}
	\T\left(L(\iv) \coprod L(\iv) \right) \ar[d] \ar[r] \ar[rrr, bend left = 15 ]
		& \lcat(\K_0)
		& & 
		\T(L(\iv) \coprod L(\iv)) \ar[d] \ar[r] 
		& (\lcat \K)_0 \arrow[from=lll, bend left = 15, crossing over]
	\\
	\T(L(C_1))	\ar[rrr, bend left = 15]
		& & &
	\T(L(C_1))	
	\end{tikzcd}
	\end{equation}	
	in which the arrow $\lcat(\K_0) \to (\lcat \K)_0$ is induced by applying $\lcat$ to the inclusion $\K_0 \hookrightarrow \K$, and the other two horizontal arrows are identities. 
\end{lemma}
\begin{proof}
	This is fairly straightforward by applying the analysis in \cite[Section 3.8]{BergerMoerdijk2012}. 
\end{proof}
The only remaining piece is to understand the morphism $\lcat(\K_0) \to (\lcat \K)_0$. 
Lemmas \ref{lem:pushout-description-interval} and \ref{lem:pushout-description-L-interval} give us a description of both $\lcat(\K_0)$ and $(\lcat \K)_0$ as pushouts of monoids in $\W$.
The following lemma will tell us that the map between them is induced by the ``obvious'' transformation of pushout diagrams.
\begin{lemma}\label{lem:monoid-pushout-transformation}
Let $f: \lcat M(C_1,C_2) \to M(L(C_1),L(C_2))$ be such that the diagram 
\begin{center}
\begin{tikzcd}
	\lcat M(C_1,C_2)
		\ar[r, "f"]
		\ar[d, "\sim" labl]
	& 
	M(L(C_1), L(C_2))
		\ar[d, "\sim" labl]
	\\
	\lcat(\K_0) 	
		\ar[r]
	&
	(\lcat \K)_0
\end{tikzcd}
\end{center}
commutes, where the left vertical arrow is the isomorphism of Theorem \ref{thm:pushout-description-K0}, and the right vertical arrow is similar. 
Then $f$ is induced by maps $f_n: \lcat M^{(n)}(C_1,C_2) \to M^{(n)}(L(C_1),L(C_2))$ which themselves are given inductively by natural transformations of diagrams
\begin{equation}\label{eqn:diagram-transformation}
\begin{tikzcd}
\T(L(X^{(n)}_{-})) \ar[d] \ar[r] \ar[rrr, bend left = 15 ]
	& \lcat M^{(n-1)}(C_1,C_2) 
	& & 
	\T(\tilde{X}^{(n)}_{-}) \ar[d] \ar[r] 
	& M^{(n-1)}(L(C_1),L(C_2)) \arrow[from=lll, bend left = 15, crossing over, "f_{n-1}"]
\\
\T\left(L(X^{(n)} )\right)	\ar[rrr, bend left = 15]
	& & &
\T\left(\tilde{X}^{(n)} \right)	
\end{tikzcd}
\end{equation}	
In which the arrows $\T(L(X^{(n)}_{-})) \to \T(\tilde{X}^{(n)}_{-})$ and $\T(L(X^{(n)})) \to \T(\tilde{X}^{(n)})$ are induced by the colax structure maps of $L$.
\end{lemma}
\begin{proof}
	The claim boils down to saying that for any $P \subseteq \n$, we have a commutative diagram 
	\begin{equation}\label{diagram:lax-pieces}
	\begin{tikzcd}
		L(X_{P}^{(n)}) 
			\ar[d] 
			\ar[r]
		&
		\tilde{X}^{(n)}_P
			\ar[d]
		\\
		\lcat M(C_1, C_2) 
			\ar[r]
		& 
		M(L(C_1), L(C_2))
	\end{tikzcd}
	\end{equation}
	In which the arrow $\lcat M(C_1,C_2) \to M(L(C_1), L(C_2))$ is such that the diagram
	\begin{center}
	\begin{tikzcd}
		\lcat M(C_1, C_2) 
			\ar[d, "\sim" labl]
			\ar[r]
		& 
		M(L(C_1),L(C_2)) 
			\ar[d, "\sim" labl]
		\\
		\lcat\left(\I(C_1,C_2)_0\right) 
			\ar[r] 
		& 
		\left(\lcat\I(C_1,C_2)\right)_0
	\end{tikzcd}
	\end{center}
	commutes. 
	The vertical arrows of diagram \eqref{diagram:lax-pieces} are those coming from the constructions of $M(C_1,C_2)$ and $M(L(C_1),L(C_2))$.
	\par 
	We show that diagram \eqref{diagram:lax-pieces} commutes in the case $n = 2$, $P = [2]$ (the other cases follow from the same logic, with more involved notation). In this case, we are showing that the diagram   
	\begin{equation}\label{diagram:lax-pieces-dim-2}
	\begin{tikzcd}
		L(\iv \ov C_2 \ov C_2 \ov \iv) 
			\ar[d] 
			\ar[r]
		&
		L(\iv) \ow L(C_2) \ow L(C_2) \ow L(\iv) 
			\ar[d]
		\\
		M(C_1, C_2) 
			\ar[r]
		& 
		M(L(C_1), L(C_2))
	\end{tikzcd}
	\end{equation}
	 commutes. 
	 The key is that we have a commutative diagram 
	 \begin{equation}
	 \begin{tikzcd}
	 	\sus(\iv \ov C_2 \ov C_2 \ov \iv) 
	 		\ar[r] 
	 		\ar[d]
		&	 	
	 	\sus(\iv, C_2, C_2, \iv)
	 		\ar[d]
	 	\\
	 	M(C_1,C_2) \ar[r]
	 	&
	 	\I(C_1,C_2)
	 \end{tikzcd}
	 \end{equation}
	 in which the left arrow comes from the construction of $M(C_1,C_2)$, and the right arrow is defined by sending:
	 	\begin{itemize}
	 		\item the first $\iv$ along the defining map $\iv \to \I(C_1,C_2)(0,1)$,
	 		\item both copies of $C_2$ along the defining map $C_2 \to \I(C_1,C_2)_1$,
	 		\item the second $\iv$ along the defining map $\iv \to \I(C_1,C_2)(1,0)$
		\end{itemize}	 		 
	 given by the map $\iv \ov C_2 \ov C_2 \ov \iv \to \I(C_1,C_2)$ which (in part) defines the isomorphism $M(C_1,C_2) \cong \I(C_1,C_2)_0$. 
	 Commutativity of this diagram comes from the definition of the isomorphism $M(C_1,C_2) \to \I(C_1,C_2)_0$.
	 Applying the functor $\lcat$ to this diagram, we obtain a diagram 
	 \begin{equation}
	 \begin{tikzcd}
	 	\sus(L(\iv \ov C_2 \ov C_2 \ov \iv)) 
	 		\ar[r] 
	 		\ar[dd]
	 	&
	 	\sus(L(\iv) \ow L(C_2) \ow L(C_2) \ow L(\iv)) 
	 		\ar[d]
		\\
		&
	 	\sus(L(\iv), L(C_2), L(C_2), L(\iv))
	 		\ar[d]
	 	\\
	 	L(M(C_1,C_2)) \ar[r]
	 	&
	 	L(\I(C_1,C_2))
	 \end{tikzcd}
	 \end{equation}
	 In which the top arrow is given by the colax structure map for $L$ by Lemma \ref{lem:join-of-intervals-is-colax-structure-map}. 
	 The result now follows from commutativity of this diagram. 
	 \par 
	 For general $n$ and $P$, the logic is identitical, though the notation is more cumbersome.
\end{proof}

\subsection{Change of enrichment Theorem \ref{alphathm:complicated}}
We are nearly ready to state and prove our second theorem on change of enrichment, which is a great deal more technical (though less restrictive) than our first. 
First, we need one definition: 
\begin{definition}[Unit Cylinder Axiom]\label{dfn:unit-cylinder-axiom}
	Let $\V$ be a monoidal model category with cofibrant unit $\iv$. 
	We say $\V$ satisfies the \emph{unit cylinder axiom} if there exists a set $\mathbb{C}$ of cylinder inclusions $\iota_i: \iv \coprod \iv \to \Cyl_i(\iv)$, with each $\iota_i$ a cofibration, such that for any object $X \in \V$:
	\begin{itemize}
		\item Any map $f: \iv \to X$ in the homotopy category of $\V$ lifts to a map $\iv \to X$ in $\V$.
		\item If two maps $f,g : \iv \to X$ represent the same map in $\Ho(\V)$, then there is a left homotopy from $f$ to $g$ along an element of $\mathbb{C}$: that is, there exists some $\iota_i: \iv \coprod \iv \to \Cyl_i(\iv)$ in $\mathbb{C}$ which fits into a commutative diagram
		\begin{center}
		\begin{tikzcd}
			\iv \coprod \iv \ar[d, "\iota_i"] \ar[r, "f \coprod g"] 
				& Y \\
			\Cyl_i(\iv) \ar[ru]
		\end{tikzcd}
		\end{center}
	\end{itemize}
\end{definition}
\begin{example}
	If all objects of $\V$ are fibrant, then $\V$ satisfies the unit cylinder axiom with $\mathbb{C}$ any cylinder object for $\iv$.
\end{example}
\begin{example}
	$\sset$ satisfies the unit cylinder axiom: the set $\mathbb{C}$ consists of those simplicial sets with geometric realization isomorphic to a line segment. 
\end{example}
The purpose of this axiom is the following consequence, which is immediate from the definitions:
\begin{lemma}\label{lem:generating-intervals-description}
	Let $\V$ be a monoidal model category with cofibrant unit $\iv$.
	If $\V$ satisfies the unit cylinder axiom (Definition \ref{dfn:unit-cylinder-axiom}) with set $\mathbb{C}$ of cylinder inclusions, then the set $\{ \mathbb{I}(C_1,C_2)\}_{C_2,C_2 \in \mathbb{C}}$ is generating in the sense of \cite[Definition 4.10]{Muro2012}. 
\end{lemma}
This stated, our last main theorem is as follows:
\begin{theorem}\label{thm:main-theorem-chapter-two}
	Let $\V$ and $\W$ be symmetric monoidal model categories. Suppose that:
	\begin{enumerate}
		\item $\V$ is combinatorial and satisfies the monoid axiom of \cite{SchwedeShipley1998}, so that in particular by \cite[Theorem 1.1]{Muro2012} the Dwyer-Kan model structure on $\vcat$ exists and is cofibrantly generated,
		\item the Dwyer-Kan model structure on $\wcat$ exists,
		\item The unit of $\V$ is cofibrant,
		\item $\V$ satisfies the Unit Cylinder Axiom (Definition \ref{dfn:unit-cylinder-axiom}).
	\end{enumerate}
	Further suppose that $L : \V \rightleftarrows \W : R$ is a Quillen adjunction, and that:
	\begin{enumerate}
	\setcounter{enumi}{4}
		\item $R$ is lax monoidal,
		\item the morphism $L(\iv) \to \iw$ induced by the corresponding colax structure on $L$ is a weak equivalence.
	\end{enumerate}
	Then the induced adjunction 
	$$\lcat : \vcat \rightleftarrows \wcat : \rcat$$
	is a Quillen adjunction.
\end{theorem}
As this list of axioms is fairly long, we will say a brief word on the necessity of each: 
\begin{enumerate}
	\item We will prove this theorem by analyzing the effect of $\lcat$ on the generating (acyclic) cofibrations of the model structure on $\vcat$ as constructed in \cite{Muro2012}, hence the necessity for $\V$ to satisfy the hypotheses there. 
	\item $\W$ does not need to satisfy all of the same axioms, though it does still need to be a monoidal model category.
	\item The assumption that the unit of $\V$ is cofibrant could likely be omitted provided one is willing to make the other axioms more complex (phrased in terms of a cofibrant replacement of the unit), but as all of the examples we are interested in have cofibrant unit we have chosen not to pursue this slightly more complex statement. 
	\item The unit-cylinder axiom is the most technical, and its assumption is very specific to our method of proof analyzing the generating acylic cofibrations of the model structure as constructed by Muro in \cite{Muro2012}. 
	Thus, we expect it is the least necessary axiom. 
	However, it is satisfied in many examples of interest: for instance, it is always satisfied if every object of $\V$ is fibrant. 
	It is also satisfied by $\sset$ and various models of spectra. 
	\item The assumption that $R$ is lax monoidal is of course necessary in the classical theorem, and thus also necessary for the homotopical analog.
	\item Finally, we have in addition to lax monoidality of $R$ the weak preservation of the unit by $L$, which appears to have no analog in the classical theorem. 
However, as we note in Section \ref{section:necessity-lax-unitality}, some condition on the unit is necessary for the change-of-enrichment adjunction to be well behaved, even when $\V$ and $\W$ have the trivial model structure. 
While the condition of weak preservation of the unit is certainly stronger than necessary, it holds for many examples of interest, and furthermore our method of proof is fairly dependent on this fact.
\end{enumerate}
We have done most of the work to prove the theorem already, as the main difficulty is in proving that the generating acylic cofibrations are sent to weak equivalences. 
\begin{proof}[Proof of \ref{thm:main-theorem-chapter-two}]
	Since $\rcat: \wcat \to \vcat$ is given by applying $R: \W \to \V$ locally, and acyclic fibrations in both $\wcat$ and $\vcat$ are exactly the local acylic fibrations which are surjective on objects, it follows immediately that $\rcat$ preserves acyclic fibrations. 
	Dually, $\lcat$ preserves cofibrations. It remains to show that $\lcat$ preserves acylic cofibrations, for which it suffices to show that $\lcat$ preserves generating acyclic cofibrations.
	\par 
	By the construction of the Dwyer-Kan model structure on $\V$ in \cite{Muro2012}, there are two types of generating acylic cofibrations for $\vcat$: the first are those of the form $\sus J$ (Definition \ref{dfn:suspension-category}) for $J$ a generating acyclic cofibration of $\V$. 
	These are preserved as a consequence of Lemma \ref{lem:catsus-preserved-by-adjoint}. 
	By Lemma \ref{lem:generating-intervals-description}, the rest of the generating acyclic cofibrations are the inclusions $\iota: \I(C_1,C_2)_0 \hookrightarrow \I(C_1,C_2)^{\cof}$ for $C_1$ and $C_2$ cylinders in our set $\mathbb{C}$, where $\I(C_1,C_2)^\cof$ is a cofibrant replacement for $\I(C_1,C_2)$ in $\vcat_{\{0,1\}}$.
	Thus in order to show that $\lcat$ preserves acyclic cofibrations, it suffices to show that the map
	$$\lcat(\iota^\cof): \lcat \left(\I(C_1,C_2)_0\right) \to \lcat\left(\I(C_1,C_2)^\cof \right)$$
	is a Dwyer-Kan equivalence. 
	Note that since $\lcat$ is a left Quillen functor $\vcat_{\{0,1\}} \to \wcat_{\{0,1\}}$ (Theorem \ref{thm:enriched-adjunction-fixed-object-set}), the map 
	$$\lcat\left(\I(C_1,C_2)^\cof \right) \to \lcat\left(\I(C_1,C_2) \right)$$
	is a weak equivalence. 
	Thus, it is equivalent to prove that
	$$\lcat(\iota): \lcat \left(\I(C_1,C_2)_0\right) \to \lcat\I(C_1,C_2) $$ 
	is a weak equivalence. 
	\par 
	Homotopical essential surjectivity is straightforward: 
	$\lcat(\iota)$ is homotopically essentially surjective iff $0 \cong 1$ in $\pi_0 \I(L(C_1), L(C_2))$.
	Since the counit $L(\iv) \to \iw$ is a weak equivalence by assumption, the maps $L(\iv) \to \I(L(C_1), L(C_2))[0,1]$ and $L(\iv) \to \I(L(C_1), L(C_2))[1,0]$ coming from the pushout description (as in diagram \eqref{diagram:pushout-description-intervals}) give inverse isomorphisms in $\pi_0 \I(L(C_1), L(C_2))$.
	\par 
	Homotopical full faithfulness will follow from the analysis of endomorphism monoids done in this section:
	homotopical full faithulness is equivalent to asking that the map
	$$\lcat(\I(C_1,C_2 )_0) \to \I(L(C_1), L(C_2))_0$$
	is a weak equivalence. 
	By Lemmas \ref{lem:pushout-description-interval}, \ref{lem:pushout-description-L-interval}, \ref{lem:last-pushout-comparison-step}, this reduces to showing that 
	$$\lcat M(C_1,C_2) \to M(L(C_1),L(C_2))$$
	is a weak equivalence, with the map being the map of Lemma \ref{lem:monoid-pushout-transformation}.
	\par
	Now, the unit axiom in $\W$ and the fact that $L$ weakly preserves the unit implies that each of the horizontal arrows in Diagram \eqref{eqn:diagram-transformation} are weak equivalences.
	Since the diagrams are both a homotopy pushouts (Lemma \ref{lem:homotopy-pushout-monoid}), it then follows that each $f_n$ is a weak equivalence. 
	Finally, note that the maps $M^{(n)}(C_1,C_2) \to M^{(n+1)}(C_1,C_2)$ are pushouts of cofibrations and hence cofibrations, so that the colimit $M(C_1,C_2) \cong \colim_n M^{(n)}(C_1,C_2)$ is in fact a homotopy colimit. Similarly for $M(L(C_1),L(C_2))$. 
	Thus, it follows that the induced map between homotopy colimits $\lcat M(C_1,C_2) \to M(L(C_1),L(C_2))$ is a weak equivalence, as desired.
\end{proof}

\section{Necessity of the lax unitality axiom}\label{section:necessity-lax-unitality}
In this section, we clarify the need for our condition (6) of Theorem \ref{thm:main-theorem-chapter-two},
which seems to have no analagous condition in the non-homotopical case.
However, this is because the classical case is somewhat misleading: even when $\V$ and $\W$ are not equipped with homotopical structure, $\vcat$ and $\wcat$ are. 
As we will show in this section, the induced adjunction does not always preserve this structure, and this failure (at least in our example) is related to the behavior of $L$ on the unit.
\par 
Let $\B$ be the category consisting of two objects, $0$ and $1$, with a unique nonidentity morphism $0 \to 1$. 
Consider the adjunction
\begin{center}
\begin{tikzcd}[sep = large]
\B \ar[r, bend left = 25, "\underline{0}"] 
& \B \ar[l, bend left = 25, "\underline{1}"]
\end{tikzcd}
\end{center}
Where $\underline{0}$ is the constant functor at the object $0$, and $\underline{1}$ is the constant functor at the object $1$. 
Then the right adjoint $\underline{1}$ is lax monoidal, and there is therefore an induced adjunction
\begin{center}
\begin{tikzcd}
\bcat \ar[r, bend left = 25, "\underline{0}^{\cat}"] 
& \bcat \ar[l, bend left = 25, "\underline{1}^\cat"]
\end{tikzcd}
\end{center}
Inspecting this adjunction, we see that the left adjoint $\underline{0}^\cat$ does not in general preserve weak equivalences: 
Let $I$ be the category with two objects, $x$ and $y$, and hom-objects given by
$$I[x,x] = I[y,y] = I[x,y] = I[y,x] = 1$$
Then one can check that 
$$\underline{0}^\cat(I)[x,x] = \underline{0}^\cat(I)[y,y] = 1 \quad \quad \quad  \underline{0}^\cat(I)[x,y] = \underline{0}^\cat(I)[y,x] = 0$$
In particular, denoting by $\iota_x^*(I)$ the full subcategory of $I$ on the object $x$, we have that $\iota_A^*(I) \hookrightarrow I$ is an equivalence, but that this equivalence is not preserved by $\underline{0}^\cat$ (because the image is not essentially surjective).
This is not due to some lack of cofibrancy; $\underline{0}^\cat$ will not preserve \textit{any} equivalences which are not bijective on objects. 
Thus, in general, we cannot expect our left adjoint $L^\cat$ to be well-behaved without some additional structure. 
\par 
We think it is an interesting question exactly what additional structure is required when $\V$ and $\W$ are ordinary categories: that is, is there a necessary and sufficient condition on the adjunction $(L,R)$ which guarantees that the adjunction $(\lcat, \rcat)$ is homotopical? 
The condition of Theorem \ref{thm:main-theorem-chapter-two} is likely stronger than is necessary.
The example above violates a much weaker condition we could ask for: that $L(\iv)$ be nonempty, i.e. admit a map $\iw \to L(\iv)$. 
This condition alone may be too weak---however, we know of no counterexamples. 
That is, we know of no examples of adjunctions $L : \V \rightleftarrows \W : R$ in which $L(\iv)$ admits a map from $\iw$, but the induced adjunction $\lcat : \vcat \rightleftarrows : \rcat$ is not Quillen with respect the Dwyer-Kan model structure.


\printbibliography
\end{document}